\documentclass[preprint,12pt]{elsarticle}

\usepackage{bm}
\usepackage{CJK}
\usepackage{enumerate}
\usepackage{color} 
\usepackage{amssymb}
\usepackage{algorithm}
\usepackage{algorithmic}
\usepackage{geometry}
\usepackage{amsfonts}
\usepackage{indentfirst}
\usepackage{mathtools}
\usepackage{setspace}
\usepackage{amsmath}
\usepackage{amsthm}

\newtheorem{theorem}{Theorem}[section]

\newtheorem{prop}[theorem]{Proposition}

\journal{Numer. Meth. Part. D. E.}

\begin{document}

\begin{frontmatter}

%% Title, authors and addresses

%% use the tnoteref command within \title for footnotes;
%% use the tnotetext command for theassociated footnote;
%% use the fnref command within \author or \address for footnotes;
%% use the fntext command for theassociated footnote;
%% use the corref command within \author for corresponding author footnotes;
%% use the cortext command for theassociated footnote;
%% use the ead command for the email address,
%% and the form \ead[url] for the home page:
%% \title{Title\tnoteref{label1}}
%% \tnotetext[label1]{}
%% \author{Name\corref{cor1}\fnref{label2}}
%% \ead{email address}
%% \ead[url]{home page}
%% \fntext[label2]{}
%% \cortext[cor1]{}
%% \affiliation{organization={},
%%             addressline={},
%%             city={},
%%             postcode={},
%%             state={},
%%             country={}}
%% \fntext[label3]{}

\title{A priori error estimates of two fully discrete coupled schemes for Biot's consolidation model}

%% use optional labels to link authors explicitly to addresses:
%% \author[label1,label2]{}
%% \affiliation[label1]{organization={},
%%             addressline={},
%%             city={},
%%             postcode={},
%%             state={},
%%             country={}}
%%
%% \affiliation[label2]{organization={},
%%             addressline={},
%%             city={},
%%             postcode={},
%%             state={},
%%             country={}}

\author[sustech,nsccsz]{Huipeng Gu}
\author[morgan]{Mingchao Cai}
\author[sustech]{Jingzhi Li}
\author[nsccsz]{Guoliang Ju}

\address[sustech]{Department of Mathematics, Southern University of Science and Technology, Shenzhen, Guangdong 518055, China}
            
\address[morgan]{Department of Mathematics, Morgan State University, Baltimore, MD 21251, USA}
            
\address[nsccsz]{National Supercomputing Center in Shenzhen, Shenzhen, Guangdong 518055, China}
            
%\address[xmu]{School of Mathematical Sciences, Xiamen University,\\Xiamen 361005, China}

\begin{abstract}
This paper concentrates on a priori error estimates of two fully discrete coupled schemes for Biot's consolidation model based on the three-field formulation introduced by Oyarz{\'u}a et al. (SIAM Journal on Numerical Analysis, 2016). The spatial discretizations are based on the Taylor–Hood finite elements combined with Lagrange elements for the three primary variables. For the time discretization, we consider two methods. One uses the backward Euler method, and the other applies a combination of the backward Euler and Crank–Nicolson methods. A priori error estimates show that the two schemes are unconditionally convergent with optimal error orders. Detailed numerical experiments are presented to validate the theoretical analysis.
\end{abstract}

%%Graphical abstract
%\begin{graphicalabstract}
%\includegraphics{grabs}
%\end{graphicalabstract}

%%Research highlights
%\begin{highlights}
%\item Research highlight 1
%\item Research highlight 2
%\end{highlights}

\begin{keyword}
%% keywords here, in the form: keyword \sep keyword
A priori error estimates \sep Biot's model \sep finite element
%% PACS codes here, in the form: \PACS code \sep code
%% MSC codes here, in the form: \MSC code \sep code
%% or \MSC[2008] code \sep code (2000 is the default)
\end{keyword}

\end{frontmatter}

%% \linenumbers

%% main text
%\section{}
%\label{}

\section{Introduction}
Biot's consolidation model \cite{biot1941general,biot1955theory} describes the interaction between
fluid flow and mechanical deformation of a porous medium, which is a solid structure
composed of pores. This model has wide applications in many scenarios including biomechanics \cite{ju2020parameter}, petroleum engineering \cite{kim2011stability}, etc. Since the exact analytic solution is hard to obtain, a variety of methods are carried out to approximate numerical solutions for the system: finite volume methods \cite{naumovich2006finite}, virtual element methods \cite{burger2021virtual}, and mixed finite element methods \cite{cai2015comparisons, lee2016robust}. %Furthermore, several approaches are proposed to solve the Biot's model, such as the implicit methods \cite{oyarzua2016locking, yi2017study}, the partitioned methods \cite{chaabane2018splitting,lee2019unconditionally}, the iterative methods \cite{both2017robust, borregales2019partially}.

In \cite{yi2017study}, it is pointed out that standard finite element methods for solving the classical two-field Biot's model may suffer from Poisson locking and pressure oscillations. Therefore, various reformulations of Biot's model are proposed to overcome these numerical difficulties. For example, fluid flux arising from the inherent Darcy law is introduced as a new variable to obtain a three-field formulation of Biot's model in \cite{yi2017study, hong2018parameter}. A new four-field formulation is proposed in \cite{yi2017iteratively}. Another three-field model is established in \cite{feng2018analysis} by introducing two pseudo-pressures. In \cite{oyarzua2016locking, lee2017parameter}, an intermediate variable, called ``total pressure", is introduced to derive a three-field reformulation of Biot's model. The advantages of such a three-field reformulation exist in that it avoids to use ${\bf H}(\mbox{div})$ space, and the classical inf-sup stable Stokes finite elements combined with Lagrange elements can be applied for the spatial discretization. By taking the advantage of such a three-field formulation, some relevant algorithms and analyses are carried out. For instance, Qi et al. \cite{qi2021finite} derive optimal-order error estimates. An unconditional convergent algorithm is proposed in \cite{lee2019unconditionally}. % based on the proposed three-field formulation. %{\color{blue} Some time-extrapolation based decoupled algorithm and iteratively decoupled algorithms are discussed in \cite{ju2020parameter, gu2022iterative}.} 

In this work, we study two fully discrete coupled schemes for solving the three-field formulation of Biot's model. In Method 1, the backward Euler method \cite{oyarzua2016locking, ju2020parameter, qi2021finite, burger2021virtual} is applied for the time discretization, which is only of first-order convergence in time. Inspired by \cite{lee2019unconditionally}, we come up with Method 2 by using a combination of the backward Euler and Crank–Nicolson methods for a second-order convergence in time. Different from the algorithm in \cite{lee2019unconditionally}, our algorithm applies a unified scheme for all time steps rather than uses different schemes for the first time step and the rest steps. Despite some existing analyses, we note that some works only assume that the Biot-Willis constant $\alpha = 1$ and the specific storage coefficient $c_0 = 0$. The only second order in time paper is \cite{lee2019unconditionally}. However, the analyses and numerical experiments there are not consistent. In this work, we conduct rigorous analysis for both Method 1 and Method 2. Optimal a priori estimates of the two methods are obtained by more careful usage of energy estimates and Gr{\"o}nwall's inequality. Furthermore, our analyses are valid for more general physical parameters. We comment here that most existing analyses only discuss %the $\bm{H}^1$ error estimates for displacement, the $L^2$ error estimate for the total pressure, and 
the $L^2$ error estimates for the fluid pressure. In this work, we also provide the $H^1$ error analysis for the fluid pressure. %In our analysis, it is shown that the two fully discrete schemes are unconditionally convergent and are of optimal error orders.

The rest of the paper is organized as follows. In Section \ref{Sec2}, we present a three-field formulation of Biot's consolidation model and the corresponding weak formulation. In Section \ref{Numer}, we introduce finite element spaces, projection operators, two fully discrete coupled schemes, and some useful propositions. A priori estimates of Method 1 and Method 2 are given in Section \ref{Sec4}. Numerical experiments are carried out to validate the theoretical results in Section \ref{Sec5}. Conclusions and outlook are given in Section \ref{conclusion}.

\section{Mathematical formulations}
\label{Sec2}
Let $\Omega \subset \mathbb{R}^d$ ($d= 2$ or $3$) be a bounded polygonal domain with boundary $\partial \Omega$. The classical Sobolev spaces are denoted by  $H^k(\Omega)$ with norm $\| \cdot \|_{H^k(\Omega)}$. We denote $H^k_{0,\Gamma} (\Omega)$ for the subspace of $H^k(\Omega )$ with the vanishing trace on $\Gamma \subset \partial \Omega$, and use $( \cdot , \cdot )$ and $\langle \cdot , \cdot \rangle$ to denote the standard $L^2(\Omega)$ and $L^2(\partial \Omega)$ inner products, respectively. In this paper, we use $C$ to denote a generic positive constant independent of mesh sizes, and use $x \lesssim y$ to denote $x \leq Cy$. The governing equations describing the quasi-static Biot system are given as follows
\begin{align}
    - \mbox{div} \sigma (\bm{u}) +\alpha \nabla p & = \bm{f}, \label{twofield1} \\
    \partial_t (c_0 p + \alpha \mbox{div}\bm{u}) - \mbox{div} K (\nabla p - \rho_f \bm{g}) & = Q_s,  \label{twofield2}
\end{align}
where
\begin{align*}
	\sigma(\bm{u})=2\mu\varepsilon(\bm{u})+\lambda (\mbox{div}\bm{u})  \bm{I}, \ \ \ \varepsilon(\bm{u}) = \frac{1}{2} \left[ \nabla \bm{u} + (\nabla \bm{u})^T \right].
\end{align*}
Here, the primary unknowns are the displacement vector of the solid $\bm{u}$ and the fluid pressure $p$, the coefficient $\alpha > 0$ is the Biot-Willis constant which is close to 1, $\bm{f}$ is the body force, $c_0 \geq 0$ is the specific storage coefficient, $K$ represents the hydraulic conductivity, $\rho_f$ is the fluid density, $\bm{g}$ is the gravitational acceleration, $Q_s$ is a source or sink term, $\bm{I}$ is the identity matrix, and Lam{\'e} constants $\lambda$ and $\mu$ are computed from the Young's modulus $E$ and the Poisson ratio $\nu$:
\begin{align*}
	\lambda = \frac{E\nu}{(1+\nu)(1-2\nu)},\ \ \ \mu = \frac{E}{2(1+\nu)}.
\end{align*}
Suitable boundary and initial conditions should be provided to complete the system. Assuming that $\partial\Omega=\Gamma_d\cup\Gamma_t=\Gamma_p\cup\Gamma_f$ with $|\Gamma_d|>0$ and $|\Gamma_p|>0$, for the simplicity of presentation, we
consider the following boundary conditions:
\begin{align}
	\bm{u} = \bm{0}       \quad & \mbox{on} \ \Gamma_d, \label{bc1} \\
	\left( \sigma(\bm{u}) - \alpha p \bm{I} \right) \bm{n} = \bm{h}  \quad & \mbox{on} \ \Gamma_t,  \\
	p=0    \quad & \mbox{on} \ \Gamma_p,  \\
	K(\nabla p - \rho_f \bm{g}) \cdot \bm{n} = g_2    \quad & \mbox{on} \ \Gamma_f, \label{bc4}
\end{align}
where $\bm{n}$ is the unit outward normal to the boundary. We comment here that the discussion can be easily extended to nonhomogeneous boundary condition cases. For the initial conditions, we consider
%Note that the method can be easily extended to nonhomogeneous cases. We consider the initial conditions:
\begin{align}
    \bm{u}(0) = \bm{u}^0, \ \ \    p(0) = p^0. \label{ic}
\end{align}

Next, we introduce an intermediate variable called ``total pressure":
$$\xi = \alpha p -\lambda \mbox{div} \bm{u}.$$ 
Then, \eqref{twofield1}-\eqref{twofield2} can be rewritten as the following three-field formulation of Biot's consolidation model
\begin{align}
    -2 \mu \mbox{div} ( \varepsilon (\bm{u})) + \nabla \xi & = \bm{f}, \label{threefield1} \\
    \mbox{div} \bm{u} + \frac{1}{\lambda} \xi - \frac{\alpha}{\lambda} p & = 0, \label{threefield2} \\
    \left( c_0 + \frac{\alpha^2}{\lambda} \right) \partial_t p - \frac{\alpha}{\lambda} \partial_t \xi - \mbox{div} K (\nabla p - \rho_f \bm{g}) & = Q_s.  \label{threefield3}
\end{align}
After the reformulation, we can still apply the boundary conditions \eqref{bc1}-\eqref{bc4} and initial conditions \eqref{ic} with $\xi(0) = \alpha p^0 - \lambda \mbox{div} \bm{u}^0$ here. For ease of presentation, we assume $\bm{g} = \bm{0}$ in the rest of the paper. 

Let $\bm{V} = \bm{H}^1_{0,\Gamma_d} (\Omega)$, $W = L^2(\Omega)$ and $M = H^1_{0,\Gamma_p} (\Omega)$, and define the bilinear forms 
\begin{align*}
    & a_1(\bm{u},\bm{v}) = 2\mu \int_\Omega  \varepsilon (\bm{u}) :  \varepsilon (\bm{v}),  \ \ \ b(\bm{v},\phi) =  \int_\Omega \phi \  \mbox{div}\bm{v}, \\
    & a_2(\xi,\phi) = \frac{1}{\lambda} \int_\Omega \xi \phi,  \quad \quad \quad \quad \quad \ c(p , \phi) =  \frac{\alpha}{\lambda} \int_\Omega p \phi, \\
    & a_3(p,\psi) = \left( c_0+\frac{\alpha^2}{\lambda} \right) \int_\Omega p \psi, \ \ \ d(p,\psi) =  K \int_\Omega \nabla p \cdot \nabla \psi.
\end{align*}
Multiplying \eqref{threefield1}-\eqref{threefield3} by test functions, integrating by parts and applying boundary conditions \eqref{bc1}-\eqref{bc4} yield the following variational problem: for a given $t \geq 0$, find $(\bm{u},\xi,p) \in \bm{V} \times W \times M$ such that
\begin{align}
    a_1(\bm{u},\bm{v})-b(\bm{v},\xi) & = ( \bm{f},\bm{v} ) + \langle \bm{h} , \bm{v} \rangle_{\Gamma_t} , \ \ \quad \quad \ \forall \bm{v} \in \bm{V}, \label{c1} \\
    b(\bm{u},\phi) + a_2(\xi,\phi) - c(p,\phi) & = 0, \ \ \quad \quad \quad \quad \quad \quad \quad \quad \quad \forall \phi \in W, \label{c2} \\
    a_3(\partial_t p,\psi)-c(\psi,\partial_t \xi)+d(p,\psi) & = ( Q_s,\psi ) + \langle g_2 , \psi \rangle_{\Gamma_f}, \quad \quad  \forall \psi \in M. \label{c3}
\end{align}
The well-posedness of problem \eqref{c1}-\eqref{c3} is established in \cite{oyarzua2016locking}. We note that the Korn's inequality \cite{korn1991nitsche} holds on $\bm{V}$, that is, there exists a constant $C_k = C_k(\Omega,\Gamma_d) > 0$ such that
\begin{align}
	\| \bm{v} \|_{H^1(\Omega)} \leq C_k \| \varepsilon(\bm{v}) \|_{L^2(\Omega)}, \quad \forall \bm{v} \in \bm{V}. \label{korncon}
\end{align}
Furthermore, the following inf-sup condition \cite{brenner1993nonconforming} holds: there exists a constant $\beta > 0$ depending only on $\Omega$ and $\Gamma_d$ such that
\begin{align*}
	\sup_{\bm{v}\in \bm{V}} \frac{  b(\bm{v},\phi) }{\| \bm{v} \|_{H^1(\Omega)}} \geq \beta \| \phi \|_{L^2(\Omega)}, \ \ \ \forall \phi \in W. %\label{infsupcon}
\end{align*}

\section{Finite element discretization and numerical schemes}
\label{Numer}
Let $\mathcal{T}_h$ be a partition of the domain $\Omega$ into triangles in $\mathbb{R}^2$ or tetrahedra in $\mathbb{R}^3$, and $h$ be the maximum diameter over all elements in the mesh. We define finite element spaces on $\mathcal{T}_h$
\begin{align*} %\label{fespace}
	& \bm{V}_{h} \coloneqq \{ \bm{v}_h \in \bm{H}^1_{0,\Gamma_d} (\Omega) \cap {\bm{C} }^0(\bar{\Omega}); \ \bm{v}_h |_E  \in {\bm P}_{k}(E), ~\forall E \in \mathcal{T}_h \},
	\\
	& W_h \coloneqq \{ \phi_h \in L^2(\Omega) \cap C^0(\bar{\Omega}); \ \phi_h |_E  \in  P_{k-1}(E), ~\forall E \in \mathcal{T}_h \},
	\\
	& M_h \coloneqq \{ \psi_h \in H^1_{0,\Gamma_p} (\Omega) \cap C^0(\bar{\Omega}); \ \psi_h |_E  \in P_{l}(E), ~\forall E \in \mathcal{T}_h  \},
\end{align*}
where $k \geq 2$ and $l \geq 1$ are two integers. In this work, the Taylor-Hood element, which consists of the pair $(\bm{V}_{h},W_h)$, and Lagrange finite element are adopted to the pair $(\bm{u},\xi)$ and $p$, respectively. Based on the three-field formulation \eqref{threefield1}-\eqref{threefield3} and discrete spaces $\bm{V}_h$, $W_h$, and $M_h$, we define two projection operators \cite{lee2019mixed,qi2021finite}. First, we introduce the Stokes projection operator $\bm{R}_{\bm{u}} \times R_{\xi} : \bm{V} \times W \rightarrow \bm{V}_h \times W_h$, which is defined by
\begin{align}
    a_1(\bm{R}_{\bm{u}} \bm{u} , \bm{v}_h) - b(\bm{v}_h, R_{\xi} \xi) & = a_1(\bm{u} , \bm{v}_h) - b(\bm{v}_h, \xi),  \quad \forall \bm{v}_h \in \bm{V}_h, \label{proj1}\\
    b(\bm{R}_{\bm{u}} \bm{u} , \phi_h) & = b( \bm{u} , \phi_h), \quad \quad \quad \quad \quad \quad \forall \phi_h \in W_h. \label{proj2}
\end{align}
 Second, we define the elliptic projection operator $R_p : M \rightarrow M_h$ as follows
\begin{align}
    d(R_p p , \psi_h) & = d(p,\psi_h), \quad \quad \quad \quad \forall \psi_h \in M_h. \label{proj3} 
\end{align}
If $\bm{u} \in \bm{H}^{k+1}_{0,\Gamma_d} (\Omega)$, $\xi \in H^{k}(\Omega)$, and $p \in H^{l+1}_{0,\Gamma_p} (\Omega)$, then the following error estimates hold true for the Stokes projection operator and the elliptic projection operator \cite{brenner2008mathematical}.
\begin{align}
    \|\bm{u}-\bm{R}_{\bm{u}} \bm{u}\|_{H^1(\Omega)} + \|\xi - R_{\xi} \xi\|_{L^2(\Omega)}
	& \leq C h^{k} (\|\bm{u}\|_{H^{{k}+1}(\Omega)} + \|\xi \|_{H^{{k}}(\Omega)}) \label{po1} , 
	\\
	\| p - R_{p} p \|_{H^1(\Omega)} & \leq C h^{l} \| p \|_{H^{{l}+1}(\Omega)}, \label{po2} 
	\\
	\| p - R_{p} p \|_{L^2(\Omega)}  & \leq C h^{l+1} \| p \|_{H^{{l}+1}(\Omega)}. \label{po3}
\end{align}
As we use a stable Stokes element pair, the corresponding finite element spaces satisfy the following discrete inf-sup condition, i.e., there exists a positive constant $\Tilde{\beta}$ independent of $h$ such that
\begin{align}
	\sup_{\bm{v}_h \in \bm{V}_h} \frac{  b(\bm{v}_h,\phi_h) }{\| \bm{v}_h \|_{H^1(\Omega)}} \geq \Tilde{\beta} \| \phi_h \|_{L^2(\Omega)}, \ \ \ \forall \phi_h \in W_h. \label{dinfsupcon}
\end{align}
%Many stable Stokes finite element pairs like Taylor-Hood elements or Mini elements can be used.
%\end{subsection}

An equidistant partition $0 = t_0 < t_1 < \cdots < t_{N+1} = T$ with a step size $\Delta t$ is considered for the time discretization. For simplicity, we introduce the notations $\bm{u}^n = \bm{u}(t_n)$, $\xi^n = \xi(t_n)$ and $p^n = p(t_n)$. Suitable approximation of initial conditions $\bm{u}^0_h = \bm{R_u} \bm{u}^0$, $\xi^0_h = R_{\xi} \xi^0$, and $p^0_h = R_p p^0$ is considered here. Following \cite{oyarzua2016locking,ju2020parameter,qi2021finite,burger2021virtual}, we present the first coupled scheme using the backward Euler method for the time discretization as follows.

\
\\
\noindent
\textbf{Method 1:} Given $(\bm{u}_h^{n},\xi_h^{n},p_h^{n}) \in \bm{V}_h \times W_h \times M_h$, find $(\bm{u}_h^{n+1},\xi_h^{n+1},p_h^{n+1}) \in \bm{V}_h \times W_h \times M_h$ such that
\begin{align}
    & a_1(\bm{u}_h^{n+1},\bm{v}_h)-b(\bm{v}_h,\xi_h^{n+1} ) = ( \bm{f}^{n+1},\bm{v}_h ) + \langle \bm{h}^{n+1} , \bm{v}_h \rangle_{\Gamma_t} , \quad  \forall \bm{v}_h \in \bm{V}_h,  \label{d1}
    \\
    & b(\bm{u}_h^{n+1},\phi_h)+a_2(\xi_h^{n+1},\phi_h) - c(p_h^{n+1},\phi_h ) = 0 %+ \chi_{ \{ n=0 \} }(n) c(\Delta t \bar{ \partial }_t p^0 ,\phi_h) 
    , \quad \quad \quad \quad \quad \ \forall \phi_h \in W_h, \label{d2}  \\
    & a_3 \left( \frac{p_h^{n+1} - p_h^{n}}{\Delta t}, \psi_h \right )-c\left( \psi_h , \frac{ \xi_h^{n+1}-\xi_h^{n} }{\Delta t} \right)
    \nonumber \\
    &  \quad \quad \quad + d(p_h^{n+1},\psi_h) = ( Q_s^{n+1},\psi_h ) + \langle g_2^{n+1} , \psi_h \rangle_{\Gamma_f}, \quad \quad \quad  \forall \psi_h \in M_h. \label{d3}
\end{align}

The second coupled scheme considers a combination of the backward Euler and Crank–Nicolson methods for time discretization, which is inspired from \cite{lee2019unconditionally}. In the entire time interval, we solve the following problem.

\
\\
\noindent
\textbf{Method 2:} Given $(\bm{u}_h^{n},\xi_h^{n},p_h^{n}) \in \bm{V}_h \times W_h \times M_h$, find $(\bm{u}_h^{n+1},\xi_h^{n+1},p_h^{n+1}) \in \bm{V}_h \times W_h \times M_h$ such that
\begin{align}
    & a_1(\bm{u}_h^{n+1},\bm{v}_h)-b(\bm{v}_h,\xi_h^{n+1} ) = ( \bm{f}^{n+1},\bm{v}_h ) + \langle \bm{h}^{n+1} , \bm{v}_h \rangle_{\Gamma_t} , \quad  \forall \bm{v}_h \in \bm{V}_h,  \label{dd1}
    \\
    & b(\bm{u}_h^{n+1},\phi_h)+a_2(\xi_h^{n+1},\phi_h) - c(p_h^{n+1},\phi_h ) = 0 , \quad \quad \quad \quad \quad \ \forall \phi_h \in W_h, \label{dd2}  \\
    & a_3 \left( \frac{p_h^{n+1} - p_h^{n}}{\Delta t}, \psi_h \right )-c\left( \psi_h , \frac{ \xi_h^{n+1}-\xi_h^{n} }{\Delta t} \right) +  d \left(\frac{p_h^{n+1} + p_h^{n}}{2},\psi_h \right) \nonumber \\
    & \quad \quad \quad = \frac{1}{2}( Q_s^{n+1} +  Q_s^{n} ,\psi_h ) + \frac{1}{2} \langle  g_2^{n+1}+g_2^{n} , \psi_h \rangle_{\Gamma_f}, \quad  \quad \ \ \forall \psi_h \in M_h. \label{dd3}
\end{align}
We comment here that such a scheme does not complicate the calculations.
 
Next, we state the following basic propositions.
\begin{prop}
Let $f:\mathbb{R} \rightarrow \mathbb{R}$ be a function that has $k + 1$ continuous derivatives on an open interval $(a,b)$. For any $t_0, t \in (a,b)$, there holds
\begin{align*}
    f(t) = f(t_0) + f'(t_0)(t-t_0) + \cdots + \frac{f^{(k)}(t_0)}{k!}(t-t_0)^k + \frac{1}{k!} \int_{t_0}^t f^{(k+1)}(s) (t-s)^k ds.
\end{align*}
Then, the following estimate for the $L^2$-norm of the last term holds true.
\begin{align}
    \left\| \frac{1}{k!} \int_{t_0}^t f^{(k+1)}(s) (t-s)^k ds \right\|_{L^2(\Omega)}^2 \lesssim (b-a)^{2k+1} \left| \int_{t_0}^t \|f^{(k+1)}\|_{L^2(\Omega)}^2 ds \right|. \label{taylor}
\end{align}
\end{prop}
\begin{proof}
The first part of the conclusion comes from the Taylor expansion theorem. Using the Cauchy-Schwarz inequality, we have 
\begin{align*}
    \left| \int_{t_0}^t f^{(k+1)}(s) (t-s)^k ds \right|^2 & \leq \int_{t_0}^t \left| t-s \right|^{2k} ds \int_{t_0}^t \left| f^{(k+1)}(s) \right|^2 ds \\
    & \leq \frac{(t-t_0)^{2k+1}}{2k+1}  \int_{t_0}^t \left| f^{(k+1)}(s) \right|^2 ds,
\end{align*}
which implies \eqref{taylor} directly.
\end{proof}
\begin{prop} \label{inequality}
Let $B$ be a symmetric bilinear form, there holds
\begin{align}
    & 2B(u,u-v) = B(u,u) - B(v,v) + B(u-v,u-v), \label{beq} \\
    & B(u+v,u-v) = B(u,u) - B(v,v), \label{beq2}
\end{align}
which immediately implies the following inequality
\begin{align}
    2B(u,u-v) \geq B(u,u) - B(v,v). \label{bineq}
\end{align}
\end{prop}

\section{Main results}
\label{Sec4}
In this section, we present our main results describing a priori error estimates of the proposed schemes. Here, we decompose error terms as
%by introducing the following notations:
\begin{align*}
	& e_{\bm{u}}^n = {\bm{u}}^n - \bm{u}_h^n = ({\bm{u}}^n - \bm{R}_{\bm{u}} {\bm{u}}^n) +( \bm{R}_{\bm{u}} {\bm{u}}^n - \bm{u}_h^n) \eqqcolon e_{\bm{u}}^{I,n} + e_{\bm{u}}^{h,n}, \\
	& e_{\xi}^n    = {\xi}^n    - \xi_h^n    = ({\xi}^n - R_{\xi}{\xi}^n) + (R_{\xi}{\xi}^n  - \xi_h^n) \eqqcolon e_{\xi}^{I,n} + e_{\xi}^{h,n}, \\
	& e_{p}^n      = {p}^n      - p_h^n      = ({p}^n - R_p {p}^n) + ( R_p {p}^n   - p_h^n) \eqqcolon e_{p}^{I,n} + e_{p}^{h,n}.
\end{align*}
We also define
\begin{align*}
    D_{\bm{u}}^{n+1} \coloneqq e_{\bm{u}}^{h,n+1} - e_{\bm{u}}^{h,n}, \ \ \ 
    D_{\xi}^{n+1} \coloneqq  e_{\xi}^{h,n+1} - e_{\xi}^{h,n}, \ \ \ 
    D_{p}^{n+1} \coloneqq e_{p}^{h,n+1} - e_{p}^{h,n}.
\end{align*}
Then, we present a priori error estimates for both schemes. The proof of each method consists of three parts. In the first part, we focus on the $\bm{H}^1$ norm of $e_{\bm{u}}^{h,N+1}$, $L^2$ norm of $e_{\xi}^{h,N+1}$, and $L^2$ norm of $e_p^{h,N+1}$. In the second part, we complete the $H^1$ norm estimate of $e_p^{h,N+1}$. In the third part, we draw our main conclusions describing the error bound. 

\subsection{A priori error estimates for Method 1.}
%{\color{blue} We not only estimate errors caused by the time-extrapolation decoupled algorithm, but also mention the results of the fully coupled algorithm derived in the first time step. Combining those two results, we finally figure out the theoretical predictions.}
%\newpage

\begin{theorem}
\label{THM}
%{\color{red}There holds}
Let $(\bm{u},\xi,p)$ and $(\bm{u}^{n+1}_h,\xi^{n+1}_h,p^{n+1}_h)$ be the solutions of Eqs. \eqref{c1}-\eqref{c3} and Eqs. \eqref{d1}-\eqref{d3}, respectively. Assume that $\bm{u} \in L^{\infty}(0,T;\bm{H}^{k+1}_{0,\Gamma_d} (\Omega))$, $\partial_{t}\bm{u} \in L^{2}(0,T ;\bm{H}^{k+1}_{0,\Gamma_d} (\Omega))$, $\partial_{tt}\bm{u} \in L^{2}(0,T;\bm{H}_{0,\Gamma_d}^{1}(\Omega))$, $\xi \in L^{\infty}(0,T;H^{k}(\Omega))$, $\partial_t \xi \in L^{2}(0,T;$ $H^{k}(\Omega))$, $\partial_{tt} \xi \in L^{2}(0,T;L^{2}(\Omega))$, $p \in L^{\infty}(0,T;H^{l+1}_{0,\Gamma_p}(\Omega))$, $\partial_{t} p \in L^{2}(0,T;H^{l+1}_{0,\Gamma_p} (\Omega))$, $\partial_{tt} p \in L^{2}(0,T;L^{2}(\Omega))$. There holds
\begin{align}
    & \| \varepsilon ( e_{\bm{u}}^{h,N+1})\|_{L^2(\Omega)}^2 + \| e_{\xi}^{h,N+1} \|_{L^2(\Omega)}^2  
    + \| e_{p}^{h,N+1} \|_{L^2(\Omega)}^2
    + \Delta t \sum_{n=0}^N \| \nabla e_p^{h,n+1} \|_{L^2(\Omega)}^2
    \nonumber \\
    & \lesssim 
    %\| \varepsilon ( e_{\bm{u}}^{h,0})\|_{L^2(\Omega)}^2 + c_0\|e_{p}^{h,0}\|_{L^2(\Omega)}^2 + \|\alpha e_{p}^{h,0} -  e_{\xi}^{h,0}\|_{L^2(\Omega)}^2 
    %\nonumber \\
    %& \ \ \ \ \ \ + 
    (\Delta t)^2 \int_{0}^{T} \left( \|\partial_{tt}\bm{u}\|^2_{H^1(\Omega)} + \|\partial_{tt}\xi \|^2_{L^2(\Omega)} + \|\partial_{tt}p \|^2_{L^2(\Omega)} \right) ds
    \nonumber \\
    & \quad + h^{2k} \int_{0}^{T}\left( \|\partial_{t} \bm{u} \|^2_{H^{k+1}(\Omega)} + \|\partial_{t} \xi \|^2_{H^{k}(\Omega)} \right)ds + h^{2l+2}  \int_{0}^{T}\|\partial_{t} p \|^2_{H^{l+1}(\Omega)}ds. \label{thmcon}
\end{align}
\end{theorem}
\begin{proof}
Setting $t=t^{n+1}$ in \eqref{c1}, \eqref{c2}, \eqref{c3}, and letting the test functions be the discrete test functions, then subtracting \eqref{d1}, \eqref{d2}, and \eqref{d3} from these equations, we immediately derive the following error equations.
\begin{align*}
    & a_1(e_{\bm{u}}^{n+1},\bm{v}_h) - b(\bm{v}_h,e_{\xi}^{n+1}) = 0,
    \\
    & b(e_{\bm{u}}^{n+1},\phi_h)+a_2(e_{\xi}^{n+1},\phi_h)
    %\nonumber \\ 
    - c( e_p^{n+1} , \phi_h ) = 0,
    \\
    & a_3\left( \partial_t p^{n+1} - \frac{p^{n+1}_h - p^{n}_h}{\Delta t},\psi_h\right) - c\left( \psi_h, \partial_t {\xi}^{n+1}- \frac{\xi^{n+1}_h - \xi^{n}_h}{\Delta t}\right) + d(e_p^{n+1},\psi_h) = 0.
\end{align*}
By using the assumptions of the projection operators \eqref{proj1}, \eqref{proj2}, and \eqref{proj3}, the above equations can be rewritten as
\begin{align}
    & a_1(e_{\bm{u}}^{h,n+1},\bm{v}_h) - b(\bm{v}_h,e_{\xi}^{h,n+1}) = 0, \label{prop2eq1}
     \\
    & b(e_{\bm{u}}^{h,n+1},\phi_h)+a_2(e_{\xi}^{n+1},\phi_h) - c( e_p^{n+1} , \phi_h ) = 0 , \label{prop2eq2}
    \\
    & a_3( D_p^{n+1},\psi_h) - c( \psi_h , D_{\xi}^{n+1}) + \Delta t d( e_p^{h,n+1}, \psi_h )  
    \nonumber \\
    & = a_3(R_p {p}^{n+1} - R_p{p}^n - \Delta t \partial_t {p}^{n+1} ,\psi_h)  - c(\psi_h,R_{\xi} {\xi}^{n+1} - R_{\xi}{\xi}^n - \Delta t \partial_t {\xi}^{n+1}). \label{prop2eq5}
\end{align}
Differentiating \eqref{c2} with respect to $t$ at the $(n+1)$-th time step, then multiplying the resulted equation by $\Delta t$, we derive that
\begin{align} 
    b( \Delta t \partial_t {\bm{u}}^{n+1} ,\phi_h) + a_2( \Delta t \partial_t {\xi}^{n+1},\phi_h) - c( \Delta t \partial_t {p}^{n+1} , \phi_h ) = 0. \label{partial}
\end{align}
For \eqref{prop2eq2}, we write out the schemes for $t=t_{n+1}$ and $t=t_n$, then take a difference between the two resulted equations, we see that 
\begin{align}
    b(D_{\bm{u}}^{n+1} ,\phi_h)+a_2(e_{\xi}^{n+1} - e_{\xi}^{n},\phi_h) - c( e_p^{n+1} - e_p^{n},\phi_h)
    = 0. \label{refo}
\end{align}
Using the definitions of $D_{\xi}^{n+1}$ and $D_p^{n+1}$, we can reformulate \eqref{refo} as follows
\begin{align}
    & b(D_{\bm{u}}^{n+1} ,\phi_h)+a_2( D_{\xi}^{n+1},\phi_h) - c( D_{p}^{n+1},\phi_h) 
    \nonumber \\
    & = - a_2( {\xi}^{n+1} - {\xi}^{n},\phi_h) + c( p^{n+1} - p^{n},\phi_h)
    \nonumber \\
    & \quad + a_2(R_{\xi} {\xi}^{n+1}-R_{\xi} {\xi}^{n},\phi_h) - c( R_p p^{n+1} - R_p p^{n},\phi_h). \label{refo1}
\end{align}
After using the $(n+1)$-th and $n$-th time steps of \eqref{c2} to obtain $b({\bm{u}}^{n+1}-  {\bm{u}}^{n}, \phi_h) = - a_2( {\xi}^{n+1} - {\xi}^{n},\phi_h) + c( p^{n+1} - p^{n},\phi_h)$, we combine \eqref{partial} and \eqref{refo1} to get
\begin{align}
    & b(D_{\bm{u}}^{n+1} ,\phi_h)+a_2(D_{\xi}^{n+1},\phi_h) - c( D_p^{n+1},\phi_h)
    = b( {\bm{u}}^{n+1}-  {\bm{u}}^{n} - \Delta t \partial_t {\bm{u}}^{n+1} ,\phi_h) 
    \nonumber \\ 
    & \ \ \ + a_2(R_{\xi} {\xi}^{n+1}-R_{\xi} {\xi}^{n}- \Delta t \partial_t {\xi}^{n+1},\phi_h)
    - c( R_p p^{n+1} - R_p p^{n}- \Delta t \partial_t {p}^{n+1} , \phi_h ) . \label{wintercon}
\end{align}
Choosing $\bm{v}_h = D_{\bm{u}}^{n+1}$ in \eqref{prop2eq1}, $\phi_h = e_{\xi}^{h,n+1}$ in \eqref{wintercon}, and $\psi_h = e_p^{h,n+1}$ in \eqref{prop2eq5}, we derive
\begin{align}
    & a_1( e_{\bm{u}}^{h,n+1} , D_{\bm{u}}^{n+1}) - b(D_{\bm{u}}^{n+1},e_{\xi}^{h,n+1}) = 0, \label{pr4eq1} \\
%\end{align}
%\begin{align}
    & b( D_{\bm{u}}^{n+1} ,e_{\xi}^{h,n+1})+a_2(D_{\xi}^{n+1},e_{\xi}^{h,n+1}) - c(D_{p}^{n+1},e_{\xi}^{h,n+1})
    \nonumber \\
    & = b( {\bm{u}}^{n+1}-  {\bm{u}}^{n} - \Delta t \partial_t {\bm{u}}^{n+1} ,e_{\xi}^{h,n+1})
    + a_2(R_{\xi} {\xi}^{n+1}-R_{\xi} {\xi}^{n}- \Delta t \partial_t {\xi}^{n+1},e_{\xi}^{h,n+1}) 
    \nonumber \\  
    & \quad - c(  R_p p^{n+1} - R_p p^{n}- \Delta t \partial_t {p}^{n+1},e_{\xi}^{h,n+1} ), \label{pr4eq11}
    \\
    & a_3( D_p^{n+1} ,e_p^{h,n+1}) - c( e_p^{h,n+1},D_{\xi}^{n+1} ) + \Delta t d(e_p^{h,n+1},e_p^{h,n+1}) 
    \nonumber \\
    & = a_3(R_p {p}^{n+1} - R_p{p}^n - \Delta t \partial_t {p}^{n+1} , e_p^{h,n+1} ) 
    \nonumber \\
    & \quad - c( e_p^{h,n+1} , R_{\xi} {\xi}^{n+1} - R_{\xi}{\xi}^n - \Delta t \partial_t {\xi}^{n+1}). \label{pr4eq2}
\end{align}
Taking the summation of \eqref{pr4eq1}, \eqref{pr4eq11}, and \eqref{pr4eq2} over the index $n$ from $0$ to $N$ yields
\begin{align}
    \text{LHS}_1 \coloneqq & \ \sum_{n=0}^N \bigg[ a_1( e_{\bm{u}}^{h,n+1} ,D_{\bm{u}}^{n+1}) +a_2(D_{\xi}^{n+1},e_{\xi}^{h,n+1}) -c(D_{p}^{n+1},e_{\xi}^{h,n+1})
    \nonumber \\
    & + a_3( D_p^{n+1} ,e_p^{h,n+1}) - c( e_p^{h,n+1},D_{\xi}^{n+1} ) + \Delta t d(e_p^{h,n+1},e_p^{h,n+1}) \bigg] = \sum_{i=1}^5 E_i, \label{LHS}
\end{align}
where
\begin{align*}
    & E_1 = \sum_{n=0}^N b( {\bm{u}}^{n+1}-  {\bm{u}}^{n} - \Delta t \partial_t {\bm{u}}^{n+1} ,e_{\xi}^{h,n+1}), \\
    & E_2 = \sum_{n=0}^N a_2(R_{\xi} {\xi}^{n+1}-R_{\xi} {\xi}^{n}- \Delta t \partial_t {\xi}^{n+1},e_{\xi}^{h,n+1}), \\ 
    & E_3 = \sum_{n=0}^N c( \Delta t \partial_t {p}^{n+1} - R_p p^{n+1} + R_p p^{n} , e_{\xi}^{h,n+1}), 
    \\
%\end{align*}
%\begin{align*}
    & E_4 = \sum_{n=0}^N a_3( R_p {p}^{n+1} - R_p{p}^n - \Delta t \partial_t {p}^{n+1} , e_p^{h,n+1} ), \\
    & E_5 = \sum_{n=0}^N c( e_p^{h,n+1} , \Delta t \partial_t {\xi}^{n+1} - R_{\xi} {\xi}^{n+1} + R_{\xi}{\xi}^n ).
\end{align*}
Using the definitions of $a_2(\cdot,\cdot)$, $a_3(\cdot,\cdot)$, and $c(\cdot,\cdot)$, we can simplify $\text{LHS}_1$ by the identity
\begin{align}
    & a_2( D_{\xi}^{n+1},e_{\xi}^{h,n+1}) -c(D_{p}^{n+1},e_{\xi}^{h,n+1})+ a_3( D_p^{n+1} ,e_p^{h,n+1}) - c( e_p^{h,n+1},D_{\xi}^{n+1})
    \nonumber \\
    & = \frac{1}{\lambda} \int_{\Omega} (\alpha D_p^{n+1} - D_{\xi}^{n+1}) (\alpha e_p^{h,n+1} - e_{\xi}^{h,n+1} ) + c_0 \int_{\Omega} (e_p^{h,n+1} - e_p^{h,n}) e_p^{h,n+1}.
    \label{Thmid}
\end{align}
Applying \eqref{bineq} and \eqref{Thmid}, we obtain the following lower bound estimate for $\text{LHS}_1$ 
\begin{align}
    & \frac{1}{2} \big( 
    2\mu \| \varepsilon ( e_{\bm{u}}^{h,N+1})\|_{L^2(\Omega)}^2 - 2\mu \| \varepsilon ( e_{\bm{u}}^{h,0})\|_{L^2(\Omega)}^2
    +
    c_0\|e_{p}^{h,N+1}\|_{L^2(\Omega)}^2 
    \nonumber \\ 
    & \quad - c_0\|e_{p}^{h,0}\|_{L^2(\Omega)}^2 
    + \frac{1}{\lambda}\|\alpha e_{p}^{h,N+1} - e_{\xi}^{h,N+1}\|_{L^2(\Omega)}^2 
    \nonumber \\
    & \quad - \frac{1}{\lambda}\|\alpha e_{p}^{h,0} - e_{\xi}^{h,0}\|_{L^2(\Omega)}^2 
    \big) 
    + K \Delta t \sum_{n=0}^N \| \nabla e_p^{h,n+1} \|_{L^2(\Omega)}^2 
%    \nonumber \\
    \leq \text{LHS}_1.
    \label{LHS2}
\end{align}

Next, we bound the terms $E_i$ for $i=1, 2, \cdots, 5$. We use the Cauchy-Schwarz inequality, the Young's inequality, \eqref{po1}, \eqref{po3} and \eqref{taylor} to estimate $E_1$, $E_2$, and $E_3$ with an $\epsilon_1 > 0$ as follows.
\begin{align*}
    E_1  & \leq 
    \frac{\epsilon_1}{6} \Delta t \sum_{n=0}^N  \|e_{\xi}^{h,n+1}\|_{L^2(\Omega)}^2 + \frac{C}{2\epsilon_1} (\Delta t)^2 \int_{0}^{T}\|\partial_{tt}\bm{u}\|^2_{H^1(\Omega)}ds,
    \nonumber \\
    E_2 & \leq 
    \frac{\epsilon_1}{6} \Delta t \sum_{n=0}^N  \|e_{\xi}^{h,n+1}\|_{L^2(\Omega)}^2 + \frac{C}{2\epsilon_1} \bigg[ (\Delta t)^2 \int_{0}^{T}\|\partial_{tt}\xi \|^2_{L^2(\Omega)}ds %+  h^{2k} \int_{t_n}^{T}\|\partial_{t} \xi \|^2_{H^{k}(\Omega)}ds \bigg],
    \\ & \quad \quad \quad \quad \quad \quad \quad \quad \quad \quad \quad \quad \quad \ + h^{2k} \int_{0}^{T} \left( \|\partial_{t} \bm{u} \|^2_{H^{k+1}(\Omega)} + \|\partial_{t} \xi \|^2_{H^{k}(\Omega)} \right) ds \bigg],
    \nonumber \\
    E_3 & \leq \frac{\epsilon_1}{6} \Delta t \sum_{n=0}^N  \|e_{\xi}^{h,n+1}\|_{L^2(\Omega)}^2 + \frac{C}{2\epsilon_1} \bigg[ (\Delta t)^2 \int_{0}^{T}\|\partial_{tt}p \|^2_{L^2(\Omega)}ds   + h^{2l+2} \int_{0}^{T} \|\partial_{t} p \|^2_{H^{l+1}(\Omega)} ds \bigg].
\end{align*}
Using the Cauchy-Schwarz inequality, the Young's inequality, the Poincar{\'e} inequality, \eqref{po1}, \eqref{po3}, and \eqref{taylor},  we can bound $E_4$ and $E_5$ with an $\epsilon_2>0$ as follows.
\begin{align*}
    E_4 & \leq 
    \frac{\epsilon_2}{4} \Delta t \sum_{n=0}^N \|\nabla e_{p}^{h,n+1}\|_{L^2(\Omega)}^2 + \frac{C}{2\epsilon_2} \bigg[ (\Delta t)^2 \int_{0}^{T}\|\partial_{tt}p \|^2_{L^2(\Omega)}ds +  h^{2l+2} \int_{0}^{T}\|\partial_{t} p \|^2_{H^{l+1}(\Omega)}ds \bigg], 
    \\
%\end{align*}
%\begin{align*}
    E_5 & \leq 
    \frac{\epsilon_2}{4} \Delta t \sum_{n=0}^N \|\nabla e_{p}^{h,n+1}\|_{L^2(\Omega)}^2 + \frac{C}{2\epsilon_2} \bigg[ (\Delta t)^2 \int_{0}^{T}\|\partial_{tt}\xi \|^2_{L^2(\Omega)}ds
    \\ & \quad \quad \quad \quad \quad \quad \quad \quad \quad \quad \quad \quad \quad \quad + h^{2k} \int_{0}^{T} \left( \|\partial_{t} \bm{u} \|^2_{H^{k+1}(\Omega)} + \|\partial_{t} \xi \|^2_{H^{k}(\Omega)} \right) ds \bigg].
\end{align*}
Combining \eqref{LHS}, \eqref{LHS2}, and the bounds $E_i$ for $i=1, 2, \cdots, 5$, we derive that
%and taking a small $\epsilon_2$ such that $\epsilon_2 \leq K$
\begin{align}
    & 2\mu \| \varepsilon ( e_{\bm{u}}^{h,N+1})\|_{L^2(\Omega)}^2 - 2\mu \| \varepsilon  (e_{\bm{u}}^{h,0}) \|_{L^2(\Omega)}^2 + c_0\|e_{p}^{h,N+1}\|_{L^2(\Omega)}^2 - c_0\|e_{p}^{h,0}\|_{L^2(\Omega)}^2 
    \nonumber \\
    & \quad + \frac{1}{\lambda}\|\alpha e_{p}^{h,N+1} - e_{\xi}^{h,N+1}\|_{L^2(\Omega)}^2 - \frac{1}{\lambda}\|\alpha e_{p}^{h,0} - e_{\xi}^{h,0}\|_{L^2(\Omega)}^2 
    + 2K \Delta t \sum_{n=0}^N \| \nabla e_p^{h,n+1} \|_{L^2(\Omega)}^2
    \nonumber \\
    & \leq \epsilon_1 \Delta t \sum_{n=0}^N \|e_{\xi}^{h,n+1}\|_{L^2(\Omega)}^2 + \epsilon_2 \Delta t \sum_{n=0}^N \|\nabla e_{p}^{h,n+1}\|_{L^2(\Omega)}^2
    % \epsilon_1 \Delta t \sum_{k=0}^n \|e_{\xi}^{h,k+1}\|_{L^2(\Omega)}^2 + 2c(\Delta t \partial_t p^0,e_{\xi}^{h,1}) + 2\sum_{k=1}^n c( D_{p}^{k}, D_{\xi}^{k+1} ) + 2c(D_{p}^{n+1}, - e_{\xi}^{h,n+1}) 
    %\nonumber \\
    %& \ \ \ \ \ \ + 
    \nonumber \\
    & \quad + \left( \frac{C}{\epsilon_1} + \frac{C}{\epsilon_2} \right) \bigg[ (\Delta t)^2 \int_{0}^{T} \left( \|\partial_{tt}\bm{u}\|^2_{H^1(\Omega)} + \|\partial_{tt}\xi \|^2_{L^2(\Omega)} + \|\partial_{tt}p \|^2_{L^2(\Omega)} \right) ds
    \nonumber \\
    & \quad + h^{2k} \int_{0}^{T}\left( \|\partial_{t} \bm{u} \|^2_{H^{k+1}(\Omega)} + \|\partial_{t} \xi \|^2_{H^{k}(\Omega)} \right)ds + h^{2l+2}  \int_{0}^{T}\|\partial_{t} p \|^2_{H^{l+1}(\Omega)}ds \bigg].
    \label{treat1}
\end{align}
Using the inf-sup condition \eqref{dinfsupcon}, \eqref{prop2eq1}, and the Korn's inequality \eqref{korncon}, we have
\begin{align}
    \|e_{\xi}^{h,n+1}\|_{L^2(\Omega)}^2 \lesssim \sup_{\bm{v}_h \in \bm{V}_h} \frac{  b(\bm{v}_h,e_{\xi}^{h,n+1}) }{\| \bm{v}_h \|_{H^1(\Omega)}} = \sup_{\bm{v}_h \in \bm{V}_h} \frac{  a_1(e_{\bm{u}}^{h,n+1},\bm{v}_h) }{\| \bm{v}_h \|_{H^1(\Omega)}} \lesssim \|\varepsilon(e_{\bm{u}}^{h,n+1})\|_{L^2(\Omega)}^2, \label{infcon}
\end{align}
which easily implies that
\begin{align}
    \|e_p^{h,n+1}\|_{L^2(\Omega)}^2 \lesssim \|\alpha e_p^{h,n+1} - e_{\xi}^{h,n+1} \|_{L^2(\Omega)}^2 + \| \varepsilon(e_{\bm{u}}^{h,n+1}) \|_{L^2(\Omega)}^2. \label{infcon2}
\end{align}
Then, we handle \eqref{treat1}. Considering the fact $\varepsilon  (e_{\bm{u}}^{h,0}) = 0$, $e_{p}^{h,0}=0$, $e_{\xi}^{h,0} = 0$, ignoring the term $c_0\|e_{p}^{h,N+1}\|_{L^2(\Omega)}^2$, using \eqref{infcon} to choose large enough positive $\epsilon_1$ such that $\epsilon_1 \|e_{\xi}^{h,k+1}\|_{L^2(\Omega)}^2 \leq 2 \mu \|\varepsilon(e_{\bm{u}}^{h,k+1})\|_{L^2(\Omega)}^2$ and setting $\epsilon_2 = K$, we can apply the discrete Gr{\"o}nwall's inequality to obtain
\begin{align}
    & 2\mu \| \varepsilon ( e_{\bm{u}}^{h,N+1})\|_{L^2(\Omega)}^2 
    % + c_0\|e_{p}^{h,N+1}\|_{L^2(\Omega)}^2  
    + \frac{1}{\lambda} \|\alpha e_{p}^{h,N+1} - e_{\xi}^{h,N+1}\|_{L^2(\Omega)}^2 
    %\nonumber \\ & \quad
    + K \Delta t \sum_{n=0}^N \| \nabla e_p^{h,n+1} \|_{L^2(\Omega)}^2
    \nonumber \\
    & \lesssim 
    %\| \varepsilon ( e_{\bm{u}}^{h,0})\|_{L^2(\Omega)}^2 + c_0\|e_{p}^{h,0}\|_{L^2(\Omega)}^2 + \|\alpha e_{p}^{h,0} -  e_{\xi}^{h,0}\|_{L^2(\Omega)}^2 
    %\nonumber \\
    %& \ \ \ \ \ \ + 
    (\Delta t)^2 \int_{0}^{T} \left( \|\partial_{tt}\bm{u}\|^2_{H^1(\Omega)} + \|\partial_{tt}\xi \|^2_{L^2(\Omega)} +  \|\partial_{tt}p \|^2_{L^2(\Omega)} \right) ds
    \nonumber \\
    & \ \ \ + h^{2k} \int_{0}^{T}\left( \|\partial_{t} \bm{u} \|^2_{H^{k+1}(\Omega)} + \|\partial_{t} \xi \|^2_{H^{k}(\Omega)} \right)ds + h^{2l+2}  \int_{0}^{T}\|\partial_{t} p \|^2_{H^{l+1}(\Omega)}ds. \label{thm411end}
\end{align}
Finally, we come to the conclusion that the desired result \eqref{thmcon} holds after applying \eqref{infcon}, \eqref{infcon2}, and \eqref{thm411end}. This completes the proof.
\end{proof}

% We comment here that Theorem \ref{THM} only provides an estimate for $\| e_p^{h,N+1} \|_{L^2(\Omega)}^2$ instead of the energy-norm estimate: $\| e_p^{h,N+1} \|_{H^1(\Omega)}^2$. In the following, we shall present an estimate for the $H^1$-norm of $e_p^{h,N+1}$.

\begin{theorem}
\label{Thm2}
Let $(\bm{u},\xi,p)$ and $(\bm{u}^{n+1}_h,\xi^{n+1}_h,p^{n+1}_h)$ be the solutions of Eqs. \eqref{c1}-\eqref{c3} and Eqs. \eqref{d1}-\eqref{d3}, respectively. Under the assumptions of \text{Theorem \ref{THM}}, there holds
\begin{align}
    & \frac{1}{\Delta t}\sum_{n=0}^N \left( \| \varepsilon (D_{\bm{u}}^{n+1}) \|_{L^2(\Omega)}^2 + \| D_{\xi}^{n+1} \|_{L^2(\Omega)}^2 + \| D_p^{n+1} \|_{L^2(\Omega)}^2 \right) + \| \nabla e_p^{h,N+1} \|_{L^2(\Omega)}^2
    \nonumber \\
    & \lesssim (\Delta t)^2 \int_{0}^{T} \left( \|\partial_{tt}\bm{u}\|^2_{H^1(\Omega)} + \|\partial_{tt}\xi \|^2_{L^2(\Omega)} + \|\partial_{tt}p \|^2_{L^2(\Omega)} \right) ds
    \nonumber \\
    & \quad + h^{2k} \int_{0}^{T}\left( \|\partial_{t} \bm{u} \|^2_{H^{k+1}(\Omega)} 
    + \|\partial_{t} \xi \|^2_{H^{k}(\Omega)} \right)ds
    %\nonumber \\
    %& \ \ \ \ \ \ 
    + h^{2l+2} \int_{0}^{T}\|\partial_{t} p \|^2_{H^{l+1}(\Omega)}ds. 
    \label{THM2}
\end{align}
\end{theorem}
\begin{proof}
Taking the difference of the $(n+1)$-th, $n$-th steps of \eqref{prop2eq1} yields
\begin{align}
    & a_1(D_{\bm{u}}^{n+1},\bm{v}_h) - b(\bm{v}_h,D_{\xi}^{n+1}) = 0, \label{prop2eq3}
\end{align}
After using the definitions of $a_2(\cdot,\cdot)$, $a_3(\cdot,\cdot)$, and $c(\cdot,\cdot)$, one has
\begin{align}
    & a_2( D_{\xi}^{n+1},D_{\xi}^{n+1}) + a_3( D_p^{n+1} ,D_p^{n+1}) - 2c( D_p^{n+1},D_{\xi}^{n+1})
    \nonumber \\
    %& \ \ \ 
    & = \frac{1}{\lambda} \| \alpha D_p^{n+1} - D_{\xi}^{n+1} \|_{L^2(\Omega)}^2 + c_0 \| D_p^{h,n+1} \|_{L^2(\Omega)}^2.
    \label{Thmid3}
\end{align}
Choosing $\bm{v}_h = D_{\bm{u}}^{n+1}$ in \eqref{prop2eq3}, $\phi_h = D_{\xi}^{n+1}$ in \eqref{wintercon}, $\psi_h = D_{p}^{n+1}$ in  \eqref{prop2eq5}, applying the identity \eqref{Thmid3}, and summing over the index $n$ from $0$ to $N$, we get
\begin{align}
    & \sum_{n=0}^N \left( 2\mu \| \varepsilon (D_{\bm{u}}^{n+1}) \|_{L^2(\Omega)}^2 + c_0 \| D_p^{n+1} \|_{L^2(\Omega)}^2 + \frac{1}{\lambda} \| \alpha D_p^{n+1} - D_{\xi}^{n+1} \|_{L^2(\Omega)}^2 \right) \nonumber \\
    & \quad + \Delta t \sum_{n=0}^N d(e_p^{h,n+1},D_{p}^{n+1}) = \sum_{i=1}^5 T_i, \label{prop2eq6}
\end{align}
where
\begin{align*}
    & T_1 = \sum_{n=0}^N b( {\bm{u}}^{n+1}-  {\bm{u}}^{n} - \Delta t \partial_t {\bm{u}}^{n+1} ,D_{\xi}^{n+1}), \\
    & T_2 = \sum_{n=0}^N a_2(R_{\xi} {\xi}^{n+1}-R_{\xi} {\xi}^{n}- \Delta t \partial_t {\xi}^{n+1},D_{\xi}^{n+1}), \\
    % \\
%\end{align*}
%\begin{align*}
    & T_3 = \sum_{n=0}^N c(\Delta t \partial_t {p}^{n+1} - R_p p^{n+1} + R_p p^{n}, D_{\xi}^{n+1}), \\
    & T_4 = \sum_{n=0}^N a_3(R_p {p}^{n+1} - R_p{p}^n - \Delta t \partial_t {p}^{n+1} ,D_{p}^{n+1}), \\
    & T_5 = \sum_{n=0}^N c(  D_{p}^{n+1}, \Delta t \partial_t {\xi}^{n+1} - R_{\xi} {\xi}^{n+1} + R_{\xi}{\xi}^n). 
\end{align*}

Next, we bound the terms $T_i$ for $i=1, 2, \cdots, 5$. Applying the Cauchy-Schwarz inequality, the Young's inequality,  \eqref{po1}, \eqref{po3} and \eqref{taylor}, we have the following estimates with an $\epsilon_1 > 0$ and an $\epsilon_2 > 0$.
\begin{align*}
    T_1 & \leq \frac{\epsilon_1}{3} \sum_{n=0}^N \|D_{\xi}^{n+1}\|_{L^2(\Omega)}^2 + \frac{C}{\epsilon_1}(\Delta t)^3 \int_{0}^{T} \|\partial_{tt}\bm{u}\|^2_{H^1(\Omega)}ds, \\
    T_2 & \leq \frac{\epsilon_1}{3} \sum_{n=0}^N \|D_{\xi}^{n+1}\|_{L^2(\Omega)}^2 +\frac{C}{\epsilon_1} \bigg[ (\Delta t)^3 \int_{0}^{T} \|\partial_{tt}\xi \|^2_{L^2(\Omega)}ds 
    \\ & \quad \quad \quad \quad \quad \quad \quad \quad \quad \quad \quad \ +  h^{2k} \Delta t \int_{0}^{T} \left( \|\partial_{t} \bm{u} \|^2_{H^{k+1}(\Omega)} + \|\partial_{t} \xi \|^2_{H^{k}(\Omega)} \right) ds \bigg], 
    \\
%\end{align*}
%\begin{align*}
    T_3 & \leq
    \frac{\epsilon_1}{3} \sum_{n=0}^N \|D_{\xi}^{n+1}\|_{L^2(\Omega)}^2 +\frac{C}{\epsilon_1}\bigg[ (\Delta t)^3 \int_{0}^{T} \|\partial_{tt}p \|^2_{L^2(\Omega)}ds + h^{2l+2} \Delta t \int_{0}^{T} \|\partial_{t} p \|^2_{H^{l+1}(\Omega)}ds \bigg],  
    \\
    T_4 & \leq \frac{\epsilon_2}{2} \sum_{n=0}^N \|D_{p}^{n+1}\|_{L^2(\Omega)}^2 + \frac{C}{\epsilon_2} \bigg[ (\Delta t)^3 \int_{0}^{T} \|\partial_{tt}p \|^2_{L^2(\Omega)}ds +  h^{2l+2} \Delta t \int_{0}^{T} \|\partial_{t} p \|^2_{H^{l+1}(\Omega)}ds \bigg], 
    \\
    T_5 & \leq \frac{\epsilon_2}{2} \sum_{n=0}^N \|D_{p}^{n+1}\|_{L^2(\Omega)}^2 + \frac{C}{\epsilon_2} \bigg[ (\Delta t)^3 \int_{0}^{T} \|\partial_{tt}\xi \|^2_{L^2(\Omega)}ds 
    \\ & \quad \quad \quad \quad \quad \quad \quad \quad \quad \quad \quad \ + h^{2k}\Delta t \int_{0}^{T} \left( \|\partial_{t} \bm{u} \|^2_{H^{k+1}(\Omega)} + \|\partial_{t} \xi \|^2_{H^{k}(\Omega)} \right) ds \bigg] .
\end{align*}
Similarly, using the inf-sup condition \eqref{dinfsupcon}, \eqref{prop2eq3}, and the Korn's inequality \eqref{korncon} yields
\begin{align}
    \|D_{\xi}^{n+1}\|_{L^2(\Omega)}^2 \lesssim \sup_{\bm{v}_h \in \bm{V}_h} \frac{  b(\bm{v}_h,D_{\xi}^{n+1}) }{\| \bm{v}_h \|_{H^1(\Omega)}} = \sup_{\bm{v}_h \in \bm{V}_h} \frac{  a_1(D_{\bm{u}}^{n+1},\bm{v}_h) }{\| \bm{v}_h \|_{H^1(\Omega)}} \lesssim \|\varepsilon(D_{\bm{u}}^{n+1})\|_{L^2(\Omega)}^2, \label{infkorn1}
\end{align}
which directly implies 
\begin{align}
    \|D_p^{n+1}\|_{L^2(\Omega)}^2 \lesssim \|\alpha D_p^{n+1} - D_{\xi}^{n+1}\|_{L^2(\Omega)}^2 + \|\varepsilon(D_{\bm{u}}^{n+1})\|_{L^2(\Omega)}^2. \label{infkorn2}
\end{align}
Using \eqref{bineq} and the fact $e_{p}^{h,0} = 0$, we obtain 
\begin{align}
    \Delta t \sum_{n=0}^N d(e_p^{h,n+1},D_{p}^{n+1}) & \geq \frac{\Delta t}{2} \sum_{n=0}^N \left[ d(e_p^{h,n+1},e_p^{h,n+1}) - d(e_p^{h,n},e_p^{h,n}) \right]
    \nonumber \\
    & \geq \frac{\Delta t}{2} d(e_p^{h,N+1},e_{p}^{h,N+1}). \label{thm422}
\end{align}
Based on \eqref{prop2eq6}, ignoring the term $c_0 \sum_{n=0}^N \| D_p^{n+1} \|_{L^2(\Omega)}^2$, using \eqref{infkorn1} and \eqref{infkorn2} to choose large enough positive $\epsilon_1$ and $\epsilon_2$ such that $ \epsilon_1 \| D_{\xi}^{n+1} \|_{L^2(\Omega)}^2$ $\leq \frac{\mu}{2} \| \varepsilon( D_{\bm{u}}^{n+1} ) \|_{L^2(\Omega)}^2$ and $ \epsilon_2 \| D_{p}^{n+1} \|_{L^2(\Omega)}^2 \leq \frac{1}{2\lambda} \|\alpha D_p^{n+1} - D_{\xi}^{n+1}\|_{L^2(\Omega)}^2 + \frac{\mu}{2} \| \varepsilon( D_{\bm{u}}^{n+1} ) \|_{L^2(\Omega)}^2 $, respectively, we can apply \eqref{thm422} and bounds of $T_i$ for $i=1, 2, \cdots, 5$ to obtain
\begin{align*}
    & \frac{1}{\Delta t}\sum_{n=0}^N \left( \mu \| \varepsilon (D_{\bm{u}}^{n+1}) \|_{L^2(\Omega)}^2 + \frac{1}{2\lambda} \| \alpha D_p^{n+1} - D_{\xi}^{n+1} \|_{L^2(\Omega)}^2 \right) +  \frac{K}{2} \|e_p^{h,N+1}\|_{L^2(\Omega)}^2
    \nonumber \\
    & \leq \left( \frac{C}{\epsilon_1} + \frac{C}{\epsilon_2} \right) \bigg[ (\Delta t)^2 \int_{0}^{T} \left( \|\partial_{tt}\bm{u}\|^2_{H^1(\Omega)} + \|\partial_{tt}\xi \|^2_{L^2(\Omega)} + \|\partial_{tt}p \|^2_{L^2(\Omega)} \right) ds
    \nonumber \\
    & \ \ \ + h^{2k} \int_{0}^{T}\left( \|\partial_{t} \bm{u} \|^2_{H^{k+1}(\Omega)} 
    + \|\partial_{t} \xi \|^2_{H^{k}(\Omega)} \right)ds
    %\nonumber \\
    %& \ \ \ \ \ \ 
    + h^{2l+2} \int_{0}^{T}\|\partial_{t} p \|^2_{H^{l+1}(\Omega)}ds \bigg].
\end{align*}
By using \eqref{infkorn1}, \eqref{infkorn2}, and the above estimate, the proof is completed.
\end{proof}

%\newpage
\begin{theorem}
\label{THM33}
Let $(\bm{u},\xi,p)$ and $(\bm{u}^{n+1}_h,\xi^{n+1}_h,p^{n+1}_h)$ be the solutions of Eqs. \eqref{c1}-\eqref{c3} and Eqs. \eqref{d1}-\eqref{d3}, respectively. Under the assumptions of \text{Theorem \ref{THM}}, there holds
\begin{align}
    \| \varepsilon ( e_{\bm{u}}^{N+1})\|_{L^2(\Omega)}   
    + \|e_{\xi}^{N+1}\|_{L^2(\Omega)}
    + \|e_{p}^{N+1}\|_{L^2(\Omega)}
    & \lesssim \Delta t + h^{k} + h^{l+1},  \\
     \| \nabla e_p^{N+1} \|_{L^2(\Omega)} & \lesssim \Delta t + h^{k} + h^{l}.
\end{align}
\end{theorem}
\begin{proof}
We start with \eqref{po1}, \eqref{po2}, and \eqref{po3}. Applying the triangle inequality, Theorem \ref{THM}, and Theorem \ref{Thm2}, 
we see that the above error estimates readily follow.
\end{proof}

% \newpage

\subsection{A priori error estimates for Method 2.}

\begin{theorem}
\label{THM3}
Let $(\bm{u},\xi,p)$ and $(\bm{u}^{n+1}_h,\xi^{n+1}_h,p^{n+1}_h)$ be the solutions of Eqs. \eqref{c1}-\eqref{c3} and Eqs. \eqref{dd1}-\eqref{dd3}, respectively. Assume that $\bm{u} \in L^{\infty}(0,T;\bm{H}^{k+1}_{0,\Gamma_d} (\Omega))$, $\partial_{t}\bm{u} \in
$ $
L^{2}(0,T;\bm{H}^{k+1}_{0,\Gamma_d} (\Omega))$, $\partial_{tt}\bm{u} \in L^{2}(0,T;\bm{H}_{0,\Gamma_d}^{1}(\Omega))$, $\partial_{ttt}\bm{u} \in L^{2}(0,T;\bm{H}_{0,\Gamma_d}^{1}(\Omega))$, $\xi \in L^{\infty}(0,T;
$ $
H^{k}(\Omega))$, $\partial_t \xi \in L^{2}(0,T;H^{k}(\Omega))$, $\partial_{tt} \xi \in L^{2}(0,T;L^{2}(\Omega))$, $\partial_{ttt} \xi \in L^{2}(0,T;L^{2}(\Omega))$, $p \in L^{\infty}(0,T;H^{l+1}_{0,\Gamma_p}(\Omega))$, $\partial_{t} p \in L^{2}(0,T;H^{l+1}_{0,\Gamma_p} (\Omega))$, $\partial_{tt} p \in L^{2}(0,T;L^{2}(\Omega))$, $\partial_{ttt} p \in L^{2}(0,T;L^{2}(\Omega))$. There holds
\begin{align}
    & \| \varepsilon ( e_{\bm{u}}^{h,N+1})\|_{L^2(\Omega)}^2 + \|e_{\xi}^{h,N+1}\|_{L^2(\Omega)}^2  
    + \| e_{p}^{h,N+1} \|_{L^2(\Omega)}^2 
    + \Delta t \sum_{n=0}^N \| \nabla ( e_p^{h,n+1} + e_p^{h,n} ) \|_{L^2(\Omega)}^2
    \nonumber \\
    & \lesssim 
    (\Delta t)^4 \int_{0}^{T} \left( \|\partial_{ttt}\bm{u}\|^2_{H^1(\Omega)} + \|\partial_{ttt}\xi \|^2_{L^2(\Omega)} +  \|\partial_{ttt}p \|^2_{L^2(\Omega)} \right) ds
    \nonumber \\
    & \quad + h^{2k} \int_{0}^{T}\left( \|\partial_{t} \bm{u} \|^2_{H^{k+1}(\Omega)} + \|\partial_{t} \xi \|^2_{H^{k}(\Omega)} \right)ds + h^{2l+2}  \int_{0}^{T}\|\partial_{t} p \|^2_{H^{l+1}(\Omega)}ds.
    \label{thm44dr}
\end{align}
\end{theorem}

\begin{proof}
Firstly, we note that \eqref{prop2eq1} and \eqref{prop2eq2} still hold here. Summing up the $(n+1)$-th, $n$-th steps of \eqref{prop2eq1}, and following a similar argument of \eqref{wintercon}, we get
\begin{align}
    & a_1(e_{\bm{u}}^{h,n+1} + e_{\bm{u}}^{h,n},\bm{v}_h) - b(\bm{v}_h,e_{\xi}^{h,n+1}+e_{\xi}^{h,n})  = 0, \label{prop2eq311} 
     \\
%\end{align}
%\begin{align}
    & b(D_{\bm{u}}^{n+1} ,\phi_h)+a_2(D_{\xi}^{n+1},\phi_h) - c( D_p^{n+1},\phi_h) 
    \nonumber \\
    & = b\left( {\bm{u}}^{n+1}-  {\bm{u}}^{n} - \frac{ \Delta t \partial_t {\bm{u}}^{n+1} + \Delta t \partial_t {\bm{u}}^{n} }{2} ,\phi_h\right) 
    \nonumber \\ 
    & \quad  + a_2 \left(R_{\xi} {\xi}^{n+1} - R_{\xi}{\xi}^n -  \frac{ \Delta t \partial_t {\xi}^{n+1} + \Delta t \partial_t {\xi}^{n} }{2},\phi_h \right) 
    \nonumber \\
    & \quad - c\left( R_p {p}^{n+1} - R_p{p}^n - \frac{ \Delta t \partial_t {p}^{n+1} + \Delta t \partial_t {p}^{n} }{2} , \phi_h \right) . \label{wintercon22}
\end{align}
After summing up the $(n+1)$-th, $n$-th time steps of \eqref{c3}, we multiply $\frac{1}{2}$ to get
\begin{align}
    & a_3 \left( \frac{\partial_t p^{n+1} +  \partial_t p^{n}}{2}, \psi_h \right )-c\left( \psi_h , \frac{\partial_t \xi^{n+1} +  \partial_t \xi^{n}}{2} \right) +  d \left(\frac{p^{n+1} + p^{n}}{2},\psi_h \right) \nonumber \\
    & = \frac{1}{2}( Q_s^{n+1} +  Q_s^{n} ,\psi_h ) + \frac{1}{2} \langle  g_2^{n+1}+g_2^{n} , \psi_h \rangle_{\Gamma_f}. \label{subs1}
\end{align}
Subtracting \eqref{dd3} from \eqref{subs1} yields
\begin{align}
    & a_3\left( \frac{\partial_t p^{n+1} +  \partial_t p^{n}}{2}  - \frac{p^{n+1}_h - p^{n}_h}{\Delta t},\psi_h\right) + d\left(  \frac{e_p^{n+1}+e_p^n}{2},\psi_h \right)
    \nonumber \\ 
    & = c\left( \psi_h, \frac{\partial_t \xi^{n+1} +  \partial_t \xi^{n}}{2}- \frac{\xi^{n+1}_h - \xi^{n}_h}{\Delta t}\right). \label{subs2}
\end{align}
Following the same argument as \eqref{prop2eq5} in Theorem \ref{THM}, we apply the projection operator \eqref{proj3} to \eqref{subs2} and derive
\begin{align}
    & a_3( D_p^{n+1},\psi_h) - c( \psi_h , D_{\xi}^{n+1}) + \frac{\Delta t}{2}  d(  e_p^{h,n+1}+e_p^{h,n},\psi_h )  \nonumber \\
    & = a_3 \left(R_p {p}^{n+1} - R_p{p}^n - \frac{ \Delta t \partial_t {p}^{n+1} + \Delta t \partial_t {p}^{n} }{2} ,\psi_h \right) 
    \nonumber \\ 
    & \quad - c\left(\psi_h,R_{\xi} {\xi}^{n+1} - R_{\xi}{\xi}^n -  \frac{ \Delta t \partial_t {\xi}^{n+1} + \Delta t \partial_t {\xi}^{n} }{2} \right). \label{prop2eq522}
\end{align}
Choosing $\bm{v}_h = D_{\bm{u}}^{n+1}$ in \eqref{prop2eq311}, $\phi_h = e_{\xi}^{h,n+1} + e_{\xi}^{h,n}$ in \eqref{wintercon22}, $\psi_h = e_{p}^{h,n+1} + e_{p}^{h,n}$ in  \eqref{prop2eq522}, and summing over the index $n$ from $0$ to $N$ yield
\begin{align}
    \text{LHS}_2 & \coloneqq \sum_{n=0}^N \bigg[ a_1(e_{\bm{u}}^{h,n+1} + e_{\bm{u}}^{h,n},D_{\bm{u}}^{n+1})+a_2(D_{\xi}^{n+1},e_{\xi}^{h,n+1} + e_{\xi}^{h,n}) 
    \nonumber \\ 
    & \quad - c( D_p^{n+1},e_{\xi}^{h,n+1} + e_{\xi}^{h,n})
    + a_3( D_p^{n+1},e_{p}^{h,n+1} + e_{p}^{h,n})
    - c( e_{p}^{h,n+1} + e_{p}^{h,n} , D_{\xi}^{n+1})
    \nonumber \\ 
    & \quad 
    + \frac{\Delta t}{2} d \left(  e_p^{h,n+1}+e_p^{h,n},e_{p}^{h,n+1} + e_{p}^{h,n} \right) \bigg] = \sum_i^5 J_i,  \label{thm4eq27}
\end{align}
where
\begin{align*}
    & J_1 = \sum_{n=0}^N b\left( {\bm{u}}^{n+1}-  {\bm{u}}^{n} - \frac{ \Delta t \partial_t {\bm{u}}^{n+1} + \Delta t \partial_t {\bm{u}}^{n} }{2} ,e_{\xi}^{h,n+1} + e_{\xi}^{h,n} \right),
    \\ 
%\end{align*}
%\begin{align*}
    & J_2 = \sum_{n=0}^N a_2 \left(R_{\xi} {\xi}^{n+1} - R_{\xi}{\xi}^n -  \frac{ \Delta t \partial_t {\xi}^{n+1} + \Delta t \partial_t {\xi}^{n} }{2},e_{\xi}^{h,n+1} + e_{\xi}^{h,n} \right), 
    \nonumber \\ 
    & J_3 = \sum_{n=0}^N c\left( \frac{ \Delta t \partial_t {p}^{n+1} + \Delta t \partial_t {p}^{n} }{2} - R_p {p}^{n+1} + R_p{p}^n ,  e_{\xi}^{h,n+1} + e_{\xi}^{h,n} \right), \\ 
    & J_4 = \sum_{n=0}^N a_3 \left(R_p {p}^{n+1} - R_p{p}^n - \frac{ \Delta t \partial_t {p}^{n+1} + \Delta t \partial_t {p}^{n} }{2} , e_{p}^{h,n+1} + e_{p}^{h,n} \right),
    \\
    & J_5 = \sum_{n=0}^N c\left( e_{p}^{h,n+1} + e_{p}^{h,n}, \frac{ \Delta t \partial_t {\xi}^{n+1} + \Delta t \partial_t {\xi}^{n} }{2} - R_{\xi} {\xi}^{n+1} + R_{\xi}{\xi}^n \right). 
\end{align*}
Using the definitions of $a_2(\cdot,\cdot)$, $a_3(\cdot,\cdot)$ and $c(\cdot,\cdot)$, we can simplify $\text{LHS}_2$ by the identity
\begin{align}
    & a_2(D_{\xi}^{n+1},e_{\xi}^{h,n+1} + e_{\xi}^{h,n}) - c( D_p^{n+1},e_{\xi}^{h,n+1} + e_{\xi}^{h,n})
    \nonumber \\
    & \quad + a_3( D_p^{n+1},e_{p}^{h,n+1} + e_{p}^{h,n}) - c( e_{p}^{h,n+1} + e_{p}^{h,n} , D_{\xi}^{n+1})
    \nonumber \\
    & = c_0 \left( \| e_{p}^{h,n+1} \|^2_{L^2(\Omega)} - \| e_{p}^{h,n} \|^2_{L^2(\Omega)} \right) 
    \nonumber \\
    & \quad + \frac{1}{\lambda} \left(\| \alpha e_{p}^{h,n+1} - e_{\xi}^{h,n+1} \|^2_{L^2(\Omega)} - \| \alpha e_{p}^{h,n} - e_{\xi}^{h,n} \|^2_{L^2(\Omega)} \right).
    \label{Thmid2}
\end{align}
Applying \eqref{thm4eq27} and \eqref{Thmid2}, we obtain
\begin{align}
    \text{LHS}_2 & = 2\mu \| \varepsilon ( e_{\bm{u}}^{h,N+1})\|_{L^2(\Omega)}^2 - 2\mu \| \varepsilon ( e_{\bm{u}}^{h,0})\|_{L^2(\Omega)}^2
    + c_0\|e_{p}^{h,N+1}\|_{L^2(\Omega)}^2 
    \nonumber \\ 
    & \quad - c_0\|e_{p}^{h,0}\|_{L^2(\Omega)}^2 + \frac{1}{\lambda} \|\alpha e_{p}^{h,N+1} - e_{\xi}^{h,N+1}\|_{L^2(\Omega)}^2 
    \nonumber \\
    & \quad - \frac{1}{\lambda}\|\alpha e_{p}^{h,0} - e_{\xi}^{h,0}\|_{L^2(\Omega)}^2 
    + \frac{K \Delta t}{2} \sum_{n=0}^N \| \nabla ( e_p^{h,n+1} + e_p^{h,n} ) \|_{L^2(\Omega)}^2. 
    \label{LHS4}
\end{align}
Assuming that $f$ is three times differentiable with respect to $t$ and $f^{'''}$ is continuous in $[0,T]$, the Taylor expansion Theorem implies
\begin{align*}
    f(t_{n+\frac{1}{2}}) & = f(t_n) + \frac{\Delta t}{2} f'(t_n) + \frac{(\Delta t)^2}{8}f''(\eta_1), \\
    f(t_{n+\frac{1}{2}}) & = f(t_{n+1}) - \frac{\Delta t}{2} f'(t_{t+1}) + \frac{(\Delta t)^2}{8}f''(\eta_2),
\end{align*}
where $\eta_1 \in (t_n,t_{n+\frac{1}{2}})$, and $\eta_2 \in (t_{n+\frac{1}{2}},t_{n+1})$. It follows that
\begin{align}
    f(t_{n+1}) - f(t_{n})  - \frac{\Delta t f'(t_{n+1}) + \Delta t f'(t_{n})}{2} = \frac{(\Delta t)^2}{8} \left[ f''(\eta_2) - f''(\eta_1) \right]. \label{thm442}
\end{align}

Next, we bound the terms $J_i$ for $i=1, 2, \cdots, 5$. Applying the Cauchy-Schwarz inequality, the Young's inequality,  \eqref{po1}, \eqref{po3}, \eqref{thm442} and \eqref{taylor}, we can bound $J_1$, $J_2$ and $J_3$ with an $\epsilon_1 > 0$ as follows
\begin{align*}
    J_1  & \leq \frac{\epsilon_1}{3} \Delta t \sum_{n=0}^N  \|e_{\xi}^{h,n+1}\|_{L^2(\Omega)}^2 + \frac{C}{\epsilon_1} (\Delta t)^4 \int_{0}^{T}\|\partial_{ttt}\bm{u}\|^2_{H^1(\Omega)}ds,
    \nonumber \\
%\end{align*}
%\begin{align*}
    J_2 & \leq 
    \frac{\epsilon_1}{3} \Delta t \sum_{n=0}^N  \|e_{\xi}^{h,n+1}\|_{L^2(\Omega)}^2 + \frac{C}{\epsilon_1} \bigg[ (\Delta t)^4 \int_{0}^{T}\|\partial_{ttt}\xi \|^2_{L^2(\Omega)}ds %+  h^{2k} \int_{t_n}^{T}\|\partial_{t} \xi \|^2_{H^{k}(\Omega)}ds \bigg],
    \\ & \quad \quad \quad \quad \quad \quad \quad \quad \quad \quad \quad \quad \ \ + h^{2k} \int_{0}^{T} \left( \|\partial_{t} \bm{u} \|^2_{H^{k+1}(\Omega)} + \|\partial_{t} \xi \|^2_{H^{k}(\Omega)} \right) ds \bigg],
    \nonumber \\
    J_3 & \leq \frac{\epsilon_1}{3} \Delta t \sum_{n=0}^N  \|e_{\xi}^{h,n+1}\|_{L^2(\Omega)}^2 + \frac{C}{\epsilon_1} \bigg[ (\Delta t)^4 \int_{0}^{T}\|\partial_{ttt}p \|^2_{L^2(\Omega)}ds   + h^{2l+2} \int_{0}^{T} \|\partial_{t} p \|^2_{H^{l+1}(\Omega)} ds \bigg].
\end{align*}
Using the Cauchy-Schwarz inequality, the Young's inequality, the Poincar{\'e} inequality, \eqref{po1}, \eqref{po3}, \eqref{thm442}, and \eqref{taylor},  we can bound $E_4$ and $E_5$ with an $\epsilon_2>0$ as follows.
\begin{align*}
    J_4 & \leq 
    \frac{\epsilon_2}{2} \Delta t \sum_{n=0}^N \|\nabla ( e_p^{h,n+1} + e_p^{h,n} ) \|_{L^2(\Omega)}^2
    \\
    & \quad + \frac{C}{\epsilon_2} \bigg[ (\Delta t)^4 \int_{0}^{T}\|\partial_{ttt}p \|^2_{L^2(\Omega)}ds +  h^{2l+2} \int_{0}^{T}\|\partial_{t} p \|^2_{H^{l+1}(\Omega)}ds \bigg], 
    \\
    J_5 & \leq 
    \frac{\epsilon_2}{2} \Delta t \sum_{n=0}^N \|\nabla ( e_p^{h,n+1} + e_p^{h,n} ) \|_{L^2(\Omega)}^2 
    \\
    & \quad + \frac{C}{\epsilon_2} \bigg[ (\Delta t)^4 \int_{0}^{T}\|\partial_{ttt}\xi \|^2_{L^2(\Omega)}ds + h^{2k} \int_{0}^{T} \left( \|\partial_{t} \bm{u} \|^2_{H^{k+1}(\Omega)} + \|\partial_{t} \xi \|^2_{H^{k}(\Omega)} \right) ds \bigg].
\end{align*}
We note that \eqref{infcon} and \eqref{infcon2} still hold here, then we deal with \eqref{thm4eq27} next. Considering the fact $\varepsilon  (e_{\bm{u}}^{h,0}) = 0$, $e_{p}^{h,0}=0$, $e_{\xi}^{h,0} = 0$, ignoring the term $c_0\|e_{p}^{h,N+1}\|_{L^2(\Omega)}^2$, using \eqref{infcon} to choose large enough positive $\epsilon_1$ such that $\epsilon_1 \|e_{\xi}^{h,k+1}\|_{L^2(\Omega)}^2 \leq 2 \mu \|\varepsilon(e_{\bm{u}}^{h,k+1})\|_{L^2(\Omega)}^2$ and setting $\epsilon_2 =  K/4$, we can apply the discrete Gr{\"o}nwall's inequality to derive
\begin{align}
    & 2\mu \| \varepsilon ( e_{\bm{u}}^{h,N+1})\|_{L^2(\Omega)}^2 
    + \frac{1}{\lambda} \|\alpha e_{p}^{h,N+1} - e_{\xi}^{h,N+1}\|_{L^2(\Omega)}^2 
    + \frac{K \Delta t}{4} \sum_{n=0}^N \| \nabla ( e_p^{h,n+1} + e_p^{h,n} ) \|_{L^2(\Omega)}^2
    \nonumber \\
    & \lesssim 
    (\Delta t)^4 \int_{0}^{T} \left( \|\partial_{ttt}\bm{u}\|^2_{H^1(\Omega)} + \|\partial_{ttt}\xi \|^2_{L^2(\Omega)} +  \|\partial_{ttt}p \|^2_{L^2(\Omega)} \right) ds
    \nonumber \\
    & \ \ \ + h^{2k} \int_{0}^{T}\left( \|\partial_{t} \bm{u} \|^2_{H^{k+1}(\Omega)} + \|\partial_{t} \xi \|^2_{H^{k}(\Omega)} \right)ds + h^{2l+2}  \int_{0}^{T}\|\partial_{t} p \|^2_{H^{l+1}(\Omega)}ds. \label{thm44eqf}
\end{align}
Finally, applying \eqref{infcon} and \eqref{infcon2} to \eqref{thm44eqf} yields the desired result \eqref{thm44dr}.
\end{proof}

\begin{theorem} \label{Thm4}
Let $(\bm{u},\xi,p)$ and $(\bm{u}^{n+1}_h,\xi^{n+1}_h,p^{n+1}_h)$ be the solutions of Eqs. \eqref{c1}-\eqref{c3} and Eqs. \eqref{dd1}-\eqref{dd3}, respectively. Under the assumptions of \text{Theorem \ref{THM3}}, there holds
\begin{align}
    & \frac{1}{\Delta t}\sum_{n=0}^N \left( \| \varepsilon (D_{\bm{u}}^{n+1}) \|_{L^2(\Omega)}^2 + \| D_{\xi}^{n+1} \|_{L^2(\Omega)}^2 + \| D_p^{n+1} \|_{L^2(\Omega)}^2 \right) + \| \nabla e_p^{h,N+1} \|_{L^2(\Omega)}^2
    \nonumber \\
    & \lesssim (\Delta t)^4 \int_{0}^{T} \left( \|\partial_{ttt}\bm{u}\|^2_{H^1(\Omega)} + \|\partial_{ttt}\xi \|^2_{L^2(\Omega)} + \|\partial_{ttt}p \|^2_{L^2(\Omega)} \right) ds
    \nonumber \\
    & \quad + h^{2k} \int_{0}^{T}\left( \|\partial_{t} \bm{u} \|^2_{H^{k+1}(\Omega)} 
    + \|\partial_{t} \xi \|^2_{H^{k}(\Omega)} \right)ds
    %\nonumber \\
    %& \ \ \ \ \ \ 
    + h^{2l+2} \int_{0}^{T}\|\partial_{t} p \|^2_{H^{l+1}(\Omega)}ds. 
\end{align}
\end{theorem}
\begin{proof}
Firstly, we note that \eqref{prop2eq1} holds here, which implies \eqref{prop2eq3} can be used here. 
%Thus, taking the difference of the $(n+1)$-th, $n$-th time steps of \eqref{prop2eq1} yields
%\begin{align}
%    & a_1(D_{\bm{u}}^{n+1},\bm{v}_h) - b(\bm{v}_h,D_{\xi}^{n+1})  = 0. \label{prop2eq32} 
%\end{align}
Choosing $\bm{v}_h = D_{\bm{u}}^{n+1}$ in  %\eqref{prop2eq32}
\eqref{prop2eq3}, $\phi_h = D_{\xi}^{n+1}$ in \eqref{wintercon22}, $\psi_h = D_{p}^{n+1}$ in \eqref{prop2eq522}, summing over the index
$n$ from $0$ to $N$, and applying the identity \eqref{Thmid3}, we can deduce that
\begin{align}
    & \sum_{n=0}^N \left( 2\mu \| \varepsilon (D_{\bm{u}}^{n+1}) \|_{L^2(\Omega)}^2 + c_0 \| D_p^{n+1} \|_{L^2(\Omega)}^2 + \frac{1}{\lambda} \| \alpha D_p^{n+1} - D_{\xi}^{n+1} \|_{L^2(\Omega)}^2 \right) 
    \nonumber \\
    & \quad + \Delta t \sum_{n=0}^N d(e_p^{h,n+1}+e_p^{h,n},D_{p}^{n+1}) 
    = \sum_{i=1}^5 L_i, \label{prop2eq7}
\end{align}
where
\begin{align*}
    & L_1 = \sum_{n=0}^N b\left( {\bm{u}}^{n+1}-  {\bm{u}}^{n} - \frac{ \Delta t \partial_t {\bm{u}}^{n+1} + \Delta t \partial_t {\bm{u}}^{n} }{2} ,D_{\xi}^{n+1} \right), \\ 
    & L_2 = \sum_{n=0}^N a_2 \left(R_{\xi} {\xi}^{n+1} - R_{\xi}{\xi}^n -  \frac{ \Delta t \partial_t {\xi}^{n+1} + \Delta t \partial_t {\xi}^{n} }{2},D_{\xi}^{n+1} \right), \\ 
    & L_3 = \sum_{n=0}^N c\left( \frac{ \Delta t \partial_t {p}^{n+1} + \Delta t \partial_t {p}^{n} }{2} - R_p {p}^{n+1} + R_p{p}^n , D_{\xi}^{n+1} \right), 
    \\
    & L_4 = \sum_{n=0}^N a_3\left(R_p {p}^{n+1} - R_p{p}^n - \frac{ \Delta t \partial_t {p}^{n+1} + \Delta t \partial_t {p}^{n} }{2} ,D_{p}^{n+1}\right),
    \\
%\end{align*}
%\begin{align*}
    & L_5 = \sum_{n=0}^N c\left( D_{p}^{n+1}, \frac{ \Delta t \partial_t {\xi}^{n+1} + \Delta t \partial_t {\xi}^{n} }{2} - R_{\xi} {\xi}^{n+1} + R_{\xi}{\xi}^n  \right).
\end{align*}

Next, we bound the terms $L_i$ for $i=1, 2, \cdots, 5$. Applying the Cauchy-Schwarz inequality, the Young's inequality,  \eqref{po1}, \eqref{po3}, \eqref{thm442}, and \eqref{taylor}, we obtain the following estimates with an $\epsilon_1 > 0$ and an $\epsilon_2 > 0$.
\begin{align*}
    L_1 & \leq \frac{\epsilon_1}{3} \sum_{n=0}^N \|D_{\xi}^{n+1}\|_{L^2(\Omega)}^2 + \frac{C}{\epsilon_1}(\Delta t)^5 \int_{0}^{T} \|\partial_{ttt}\bm{u}\|^2_{H^1(\Omega)}ds, \\
    L_2 & \leq \frac{\epsilon_1}{3} \sum_{n=0}^N \|D_{\xi}^{n+1}\|_{L^2(\Omega)}^2 +\frac{C}{\epsilon_1} \bigg[ (\Delta t)^5 \int_{0}^{T} \|\partial_{ttt}\xi \|^2_{L^2(\Omega)}ds 
    \\ & \quad \quad \quad \quad \quad \quad \quad \quad \quad \quad \quad \quad +  h^{2k} \Delta t \int_{0}^{T} \left( \|\partial_{t} \bm{u} \|^2_{H^{k+1}(\Omega)} + \|\partial_{t} \xi \|^2_{H^{k}(\Omega)} \right) ds \bigg], 
    \\
    L_3 & \leq
    \frac{\epsilon_1}{3} \sum_{n=0}^N \|D_{\xi}^{n+1}\|_{L^2(\Omega)}^2 +\frac{C}{\epsilon_1}\bigg[ (\Delta t)^5 \int_{0}^{T} \|\partial_{ttt}p \|^2_{L^2(\Omega)}ds + h^{2l+2} \Delta t \int_{0}^{T} \|\partial_{t} p \|^2_{H^{l+1}(\Omega)}ds \bigg],  
    \\
    L_4 & \leq \frac{\epsilon_2}{2} \sum_{n=0}^N \|D_{p}^{n+1}\|_{L^2(\Omega)}^2 + \frac{C}{\epsilon_2} \bigg[ (\Delta t)^5 \int_{0}^{T} \|\partial_{ttt}p \|^2_{L^2(\Omega)}ds +  h^{2l+2} \Delta t \int_{0}^{T} \|\partial_{t} p \|^2_{H^{l+1}(\Omega)}ds \bigg], 
%\end{align*}
%\begin{align*}
    \\
    L_5 & \leq \frac{\epsilon_2}{2} \sum_{n=0}^N \|D_{p}^{n+1}\|_{L^2(\Omega)}^2 + \frac{C}{\epsilon_2} \bigg[ (\Delta t)^5 \int_{0}^{T} \|\partial_{ttt}\xi \|^2_{L^2(\Omega)}ds 
    \\ & \quad \quad \quad \quad \quad \quad \quad \quad \quad \quad \quad \quad   + h^{2k}\Delta t \int_{0}^{T} \left( \|\partial_{t} \bm{u} \|^2_{H^{k+1}(\Omega)} + \|\partial_{t} \xi \|^2_{H^{k}(\Omega)} \right) ds \bigg] .
\end{align*}
We note that \eqref{infkorn1} and \eqref{infkorn2} still hold here, and then we handle \eqref{prop2eq7} next. Using the fact $e_{p}^{h,0} = 0$, ignoring the term $c_0 \sum_{n=0}^N \| D_{p}^{n+1} \|^2_{L^2(\Omega)}$, using \eqref{infkorn1} and \eqref{infkorn2} to choose large enough positive $\epsilon_1$ and $\epsilon_2$ such that $\epsilon_1 \| D_{\xi}^{n+1} \|_{L^2(\Omega)}^2 \leq \frac{\mu}{2} \| \varepsilon( D_{\bm{u}}^{n+1} ) \|_{L^2(\Omega)}^2$ and $ \epsilon_2 \| D_{p}^{n+1} \|_{L^2(\Omega)}^2 \leq \frac{1}{2\lambda} \|\alpha D_p^{n+1} - D_{\xi}^{n+1}\|_{L^2(\Omega)}^2 + \frac{\mu}{2} \| \varepsilon( D_{\bm{u}}^{n+1} ) \|_{L^2(\Omega)}^2 $, respectively, we obtain
\begin{align*}
    & \frac{1}{\Delta t} \sum_{n=0}^N \left( \mu \| \varepsilon (D_{\bm{u}}^{n+1}) \|^2_{L^2(\Omega)} + \frac{1}{2\lambda} \| \alpha D_{p}^{n+1} - D_{\xi}^{n+1} \|^2_{L^2(\Omega)} \right)
    + K \| \nabla e_{p}^{h,N+1} \|^2_{L^2(\Omega)}  \nonumber \\
    & \leq \left( \frac{C}{\epsilon_1} + \frac{C}{\epsilon_2} \right) \bigg[ (\Delta t)^4 \int_{0}^{T} \left( \|\partial_{ttt}\bm{u}\|^2_{H^1(\Omega)} + \|\partial_{ttt}\xi \|^2_{L^2(\Omega)} + \|\partial_{ttt}p \|^2_{L^2(\Omega)} \right) ds \nonumber \\ 
    & \quad + h^{2k} \int_{0}^{T} \left( \|\partial_{t} \bm{u} \|^2_{H^{k+1}(\Omega)} + \|\partial_{t} \xi \|^2_{H^{k}(\Omega)} \right) ds
    + h^{2l+2} \int_{0}^{T} \|\partial_{t} p \|^2_{H^{l+1}(\Omega)}ds \bigg].
\end{align*}
Applying \eqref{infkorn1} and \eqref{infkorn2} to the above equation, we claim that the proof is completed.
\end{proof}

\begin{theorem}
\label{THM66}
%{\color{red}There holds}
Let $(\bm{u},\xi,p)$ and $(\bm{u}^{n+1}_h,\xi^{n+1}_h,p^{n+1}_h)$ be the solutions of Eqs. \eqref{c1}-\eqref{c3} and Eqs. \eqref{dd1}-\eqref{dd3}, respectively. Under the assumptions of \text{Theorem \ref{THM3}}, there holds
\begin{align}
    \| \varepsilon ( e_{\bm{u}}^{N+1})\|_{L^2(\Omega)}   
    + \|e_{\xi}^{N+1}\|_{L^2(\Omega)}
    + \|e_{p}^{N+1}\|_{L^2(\Omega)}
    & \lesssim (\Delta t)^2 + h^{k} + h^{l+1},  \\
     \| \nabla e_p^{N+1} \|_{L^2(\Omega)} & \lesssim (\Delta t)^2 + h^{k} + h^{l}.
\end{align}
\end{theorem}
\begin{proof}
We start with \eqref{po1}, \eqref{po2}, and \eqref{po3}. Applying the triangle inequality, Theorem \ref{THM3}, and Theorem \ref{Thm4}, 
we see that the above error estimates readily follow.  
\end{proof}

% \newpage

\section{Benchmark tests} \label{Sec5}
In this section, we present numerical experiments in two dimensions to validate the theoretical predictions described in Section \ref{Sec4}. %Since several tests are presented in \cite{ju2020parameter}, here we consider tests with analytical solutions \cite{chaabane2018splitting,feng2018analysis,yi2017study}, and focus on the convergence rates under various settings of $\Delta t$, $k_{\bm{u}}$ and $k_p$. 
All computations are implemented by using the open-source software FreeFEM++  \cite{2013New}.

\
\\
\noindent
{\bf Example 1.} Let the domain $\Omega=[0,1]^2$ and the final time is $T=1.0$. We choose the body force $\bm{f}$, the source/sink term $Q_s$, initial conditions and Dirichlet boundary data on $\partial \Omega = \Gamma_d = \Gamma_p$ such that the exact solution is as follows:
\begin{align*}
    u_1 = \frac{1}{10} e^t (x+y^3), \quad
    u_2 = \frac{1}{10} t^2 (x^3+y^3), \quad
    p = 10e^{\frac{x+y}{10}}(1+t^3).
\end{align*}
Following \cite{chaabane2018splitting}, the physical parameters are given by:
\begin{align*}
    \mu = 1.0, \quad \lambda = 1.0, \quad c_0 = 1.0, \quad \alpha = 1.0, \quad K = 1.0.
\end{align*}
We apply a small mesh size $h = \frac{1}{64}$ and take polynomial orders $k = 3$, $l = 2$ for the spatial discretization so that the spatial error is not dominant. To check the orders of convergence in time, we only refine the time step size $\Delta t$. In Table \ref{E11} and Table \ref{E12}, we present the results of errors and convergences rates for Method 1 and Method 2, respectively. We observe that the orders of $\bm{H}^1$ error of $\bm{u}$, $L^2$ error of $\xi$, $L^2$ and $H^1$ errors of $p$ are all around $1$ in Table \ref{E11}, and are all around $2$ in Table \ref{E12}. The results in both tables illustrate that the time error order based on Method 1 is $\mathcal{O} (\Delta t)$ and the the time error order based on Method 2 is $\mathcal{O} ((\Delta t)^2)$,  which verify the theoretical predictions of error analyses in Theorem \ref{THM33} and Theorem \ref{THM66}.

\begin{table}[H]
\begin{center}
\caption{Errors and convergence rates given by Method 1 for Example 1.}
\label{E11}
\centering
{\scriptsize
	\begin{tabular}{cclclcl}
	\hline
	$\Delta t$ & \multicolumn{1}{l}{$\bm{H}^1$ errors of $\bm{u}$} & \multicolumn{1}{l}{Orders} & \multicolumn{1}{l}{$L^2$ errors of $\xi$} & \multicolumn{1}{l}{Orders} & \multicolumn{1}{l}{$L^2$\& $H^1$ errors of $p$} & \multicolumn{1}{l}{Orders} \\ \hline
    1/4 & 5.219e-02 &       & 2.754e-01 &      & 2.971e-01 \& 1.386e+00 &  \\
    1/8 & 2.735e-02 &  0.93 & 1.443e-01 & 0.93 & 1.557e-01 \& 7.263e-01 & 0.93 \& 0.93 \\
    1/16 & 1.399e-02 & 0.97 & 7.381e-02 & 0.97 & 7.963e-02 \& 3.715e-01 & 0.97 \& 0.97 \\
    1/32 & 7.076e-03 & 0.98 & 3.732e-02 & 0.98 & 4.026e-02 \& 1.878e-01 & 0.98 \& 0.98 \\\hline
	\end{tabular}
}
\end{center}
\end{table}

\begin{table}[H]
\begin{center}
\caption{Errors and convergence rates given by Method 2 for Example 1.}
\label{E12}
\centering
{\scriptsize
	\begin{tabular}{cclclcl}
	\hline
	$\Delta t$ & \multicolumn{1}{l}{$\bm{H}^1$ errors of $\bm{u}$} & \multicolumn{1}{l}{Orders} & \multicolumn{1}{l}{$L^2$ errors of $\xi$} & \multicolumn{1}{l}{Orders} & \multicolumn{1}{l}{$L^2$\& $H^1$ errors of $p$} & \multicolumn{1}{l}{Orders} \\ \hline
    1/4 & 2.630e-03  &      & 1.266e-02 &      & 1.385e-02 \& 6.333e-02 &  \\
    1/8 & 6.426e-04  & 2.03 & 3.296e-03 & 1.94 & 3.570e-03 \& 1.653e-02 & 1.96 \& 1.94 \\
    1/16 & 1.587e-04 & 2.02 & 8.278e-04 & 1.99 & 8.944e-04 \& 4.159e-03 & 2.00 \& 1.99 \\
    1/32 & 3.959e-05 & 2.00 & 2.071e-04 & 2.00 & 2.237e-04 \& 1.041e-03 & 2.00 \& 2.00 \\ \hline
	\end{tabular}
}
\end{center}
\end{table}

\
\\
\noindent
{\bf Example 2.} Let the domain $\Omega=[0,1]^2$ with $\Gamma_1=\{(1,y); 0\leq y\leq1\}$, $\Gamma_2=\{(x,0);0\leq x\leq1\}$,
$\Gamma_3=\{(0,y);0\leq y\leq1\}$,  $\Gamma_4=\{(x,1);0\leq x\leq1\}$ and the final time is $T=1.0$. The Neumann boundary $\Gamma_t = \Gamma_f = \Gamma_1 \cup \Gamma_3$ and the Dirichlet boundary $\Gamma_d = \Gamma_p = \Gamma_2 \cup \Gamma_4$ are considered in this example. We take the body force $\bm{f}$, the source/sink term $Q_s$, initial and boundary conditions such that the exact solution is as follows:
\begin{align*}
    u_1 & = e^{-t} \left( \sin{(2\pi y)} (\cos{(2\pi x)}-1) + \frac{1}{\mu+\lambda} \sin{(\pi x)}\sin{(\pi y)} \right), \\
    u_2 & = e^{-t} \left( \sin{(2\pi x)} (1-\cos{(2\pi y)}) + \frac{1}{\mu+\lambda} \sin{(\pi x)}\sin{(\pi y)} \right), \\
    p & = e^{-t}\sin{(\pi x)}\sin{(\pi y)}.
\end{align*}
The fixed physical parameters are:
\begin{align*}
    E = 1.0, \quad c_0 = 1.0, \quad \alpha = 1.0.
\end{align*}
Other physical parameters will vary to test the robustness of our numerical schemes. Numerical results for this example are summarized in Tables \ref{E21}-\ref{E34}. Among them, we examine Method 1 in Tables \ref{E21}, \ref{E22}, \ref{E31}, \ref{E32}, and Method 2 in Tables \ref{E23}, \ref{E24}, \ref{E33}, \ref{E34}. Since we have already verified the error orders in time for both schemes in Example 1, our focus here is the verification of the spatial error orders. To verify the spatial error orders as analyzed in Theorem \ref{THM33}, we take $\Delta t$ of an order $\mathcal{O}(h^2)$ (Table \ref{E21} and Table \ref{E31}) or $\mathcal{O}(h^3)$ (Table \ref{E22} and Table \ref{E32}) for Method 1. Similarly, to verify the spatial error orders as analyzed in Theorem \ref{THM66}, we take $\Delta t$ of an order $\mathcal{O}(h)$ (Table \ref{E23} and Table \ref{E33}) or $\mathcal{O}(h^2)$ (Table \ref{E24} and Table \ref{E34}) for Method 2. 

%\newpage
%theoretical analysis for spatial error orders. 
Firstly, we fix $\nu = 0.3$ and $K = 1.0$. The numerical results for errors and convergence orders using $k=2$, $l=1$ and $k=3$, $l=2$ are presented in Table \ref{E21}, \ref{E23} and Table \ref{E22}, \ref{E24}, respectively. We refine the mesh size (and vary the corresponding time step size) to present the numerical results. 
From Table \ref{E21}, it is clearly shown that the convergence $\|e_{\bm{u}}^{N+1}\|_{H^1(\Omega)}$, $\|e_{\xi}^{N+1}\|_{L^2(\Omega)}$, $\|e_p^{N+1}\|_{L^2(\Omega)}$ are of order $\mathcal{O}(\Delta t + h^2)$, and $\|e_p^{N+1}\|_{H^1(\Omega)}$ is of order $\mathcal{O}(\Delta t + h)$ for Method 1. Similarly, from Table \ref{E22}, we see that the convergence $\|e_{\bm{u}}^{N+1}\|_{H^1(\Omega)}$, $\|e_{\xi}^{N+1}\|_{L^2(\Omega)}$, $\|e_p^{N+1}\|_{L^2(\Omega)}$ are of order $\mathcal{O}(\Delta t + h^3)$, and $\|e_p^{N+1}\|_{H^1(\Omega)}$ is of order $\mathcal{O}(\Delta t + h^2)$. 
From Table \ref{E23}, we see that the convergence $\|e_{\bm{u}}^{N+1}\|_{H^1(\Omega)}$, $\|e_{\xi}^{N+1}\|_{L^2(\Omega)}$, $\|e_p^{N+1}\|_{L^2(\Omega)}$ are of order $\mathcal{O}((\Delta t)^2 + h^2)$, and $\|e_p^{N+1}\|_{H^1(\Omega)}$ is of order $\mathcal{O}((\Delta t)^2 + h)$ for Method 2. Moreover, from Table \ref{E24}, we see that the convergence $\|e_{\bm{u}}^{N+1}\|_{H^1(\Omega)}$, $\|e_{\xi}^{N+1}\|_{L^2(\Omega)}$, $\|e_p^{N+1}\|_{L^2(\Omega)}$ are of order $\mathcal{O}((\Delta t)^2 + h^3)$, and $\|e_p^{N+1}\|_{H^1(\Omega)}$ is of order $\mathcal{O}((\Delta t)^2 + h^2)$ .

%Similarly, from Table 4, we see that the spatial error orders are also optimal if one adopt a higher-order finite element for each variable. The results from Table 3 and Table 4 verify the optimal spatial convergence analysis in Theorems \ref{THM33}.  
%{\color{blue}
%For Method 2, from Table 5,  it is clearly shown that the energy-norm error orders for $\bm{u}$ and $\xi$ are $\mathcal{O}(h^3)$ and the energy-norm error order for $p$ is $\mathcal{O}(h)$. Similarly, from Table 6, we see that the spatial error orders are also optimal if one adopt a higher-order finite element for each variable. The results from Table 5 and Table 6 verify the optimal spatial convergence analysis in Theorem \ref{THM66}.}

\begin{table}[H]
\begin{center}
\caption{Errors and convergence rates given by Method 1 for Example 2 using $k=2$ and $l=1$ with $\nu = 0.3$ and $K=1$.}
\label{E21}
\centering
{\scriptsize
	\begin{tabular}{ccclclcl}
	\hline
	$h$  &  $\Delta t$  & \multicolumn{1}{l}{ $\bm{H}^1$ errors of $\bm{u}$} & \multicolumn{1}{l}{Orders} & \multicolumn{1}{l}{$L^2$ errors of $\xi$} & \multicolumn{1}{l}{Orders} & \multicolumn{1}{l}{$L^2$\& $H^1$ errors of $p$} & \multicolumn{1}{l}{Orders} \\ \hline
    1/2 & 1/4    & 1.715e+00 &  & 8.136e-02 &  & 7.877e-02 \& 6.165e-01 &  \\
    1/4 & 1/16   & 5.207e-01 & 1.72 & 3.801e-02 & 1.10 & 2.475e-02 \& 3.398e-01 & 1.67 \& 0.86 \\
    1/8 & 1/64   & 1.436e-01 & 1.86 & 8.062e-03 & 2.24 & 6.939e-03 \& 1.774e-01 & 1.83 \& 0.94 \\
    1/16 & 1/256 & 3.713e-02 & 1.95 & 1.894e-03 & 2.09 & 1.805e-03 \& 8.992e-02 & 1.94 \& 0.98 \\ \hline
	\end{tabular}
}
\end{center}
\end{table}
\begin{table}[H]
\begin{center}
\caption{Errors and convergence rates given by Method 1 for Example 2 using $k=3$ and $l=2$ with $\nu = 0.3$ and $K=1$.}
\label{E22}
\centering
{\scriptsize
	\begin{tabular}{ccclclcl}
	\hline
	$h$ & $\Delta t$                    & \multicolumn{1}{l}{ $\bm{H}^1$ errors of $\bm{u}$} & \multicolumn{1}{l}{Orders} & \multicolumn{1}{l}{$L^2$ errors of $\xi$} & \multicolumn{1}{l}{Orders} & \multicolumn{1}{l}{$L^2$\& $H^1$ errors of $p$} & \multicolumn{1}{l}{Orders} \\ \hline
    1/2 & 1/8     & 5.685e-01 &  & 4.381e-02 &  & 5.164e-03 \& 6.914e-02 &  \\
    1/4 & 1/64    & 7.997e-02 & 2.83 & 4.889e-03 & 3.16 & 1.034e-03 \& 3.271e-02 & 2.32 \& 1.08 \\
    1/8 & 1/512   & 9.833e-03 & 3.02 & 5.956e-04 & 3.04 & 1.364e-04 \& 9.259e-03 & 2.92 \& 1.82 \\
    1/16 & 1/4096 & 1.205e-03 & 3.03 & 7.267e-05 & 3.03 & 1.756e-05 \& 2.423e-03 & 2.96 \& 1.93 \\
    \hline
	\end{tabular}
}
\end{center}
\end{table}
\begin{table}[H]
\begin{center}
\caption{Errors and convergence rates given by Method 2 for Example 2 using $k=2$ and $l=1$ with $\nu = 0.3$ and $K=1$.}
\label{E23}
\centering
{\scriptsize
	\begin{tabular}{ccclclcl}
	\hline
	$h$ & $\Delta t$              & \multicolumn{1}{l}{ $\bm{H}^1$ errors of $\bm{u}$} & \multicolumn{1}{l}{Orders} & \multicolumn{1}{l}{$L^2$ errors of $\xi$} & \multicolumn{1}{l}{Orders} & \multicolumn{1}{l}{$L^2$\& $H^1$ errors of $p$} & \multicolumn{1}{l}{Orders} \\ \hline
    1/2 & 1/2   & 1.716e+00 &      & 8.492e-02 &      & 9.235e-02 \& 9.311e-01 &  \\
    1/4 & 1/4   & 5.206e-01 & 1.72 & 3.729e-02 & 1.19 & 2.195e-02 \& 3.683e-01 & 2.07 \& 1.34 \\
    1/8 & 1/8   & 1.436e-01 & 1.86 & 8.080e-03 & 2.21 & 6.977e-03 \& 1.802e-01 & 1.65 \& 1.03 \\
    1/16 & 1/16 & 3.713e-02 & 1.95 & 1.905e-03 & 2.08 & 1.828e-03 \& 9.023e-02 & 1.93 \& 1.00 \\ \hline
	\end{tabular}
}
\end{center}
\end{table}
\begin{table}[H]
\begin{center}
\caption{Errors and convergence rates given by Method 2 for Example 2 using $k=3$ and $l=2$ with $\nu = 0.3$ and $K=1$.}
\label{E24}
\centering
{\scriptsize
	\begin{tabular}{ccclclcl}
	\hline
	$h$ & $\Delta t$                     & \multicolumn{1}{l}{ $\bm{H}^1$ errors of $\bm{u}$} & \multicolumn{1}{l}{Orders} & \multicolumn{1}{l}{$L^2$ errors of $\xi$} & \multicolumn{1}{l}{Orders} & \multicolumn{1}{l}{$L^2$\& $H^1$ errors of $p$} & \multicolumn{1}{l}{Orders} \\ \hline
    1/2 & 1/4    & 5.686e-01 &      & 4.377e-02 &      & 7.038e-03 \& 1.112e-01 &  \\
    1/4 & 1/16   & 7.997e-02 & 2.83 & 4.872e-03 & 3.17 & 9.834e-04 \& 3.449e-02 & 2.84 \& 1.69 \\
    1/8 & 1/64   & 9.833e-03 & 3.02 & 5.953e-04 & 3.03 & 1.312e-04 \& 9.258e-03 & 2.91 \& 1.90 \\
    1/16 & 1/256 & 1.205e-03 & 3.03 & 7.262e-05 & 3.04 & 1.688e-05 \& 2.423e-03 & 2.96 \& 1.93 \\ \hline
	\end{tabular}
}
\end{center}
\end{table}
\noindent

Secondly, we fix $\nu = 0.49999$ and $K = 10^{-6}$ to test the robustness of the proposed schemes with respect to the key physical parameters. The numerical results for errors and convergence orders using $k=2$, $l=1$ and $k=3$, $l=2$ are presented in Table \ref{E31}, \ref{E33} and Table \ref{E32}, \ref{E34}, respectively.
By checking the error results and convergence rates one table by one table, one can verify the theoretical analysis provided in Section 4. From these tables, it is shown clearly that all energy norm errors decrease with the optimal convergence orders.
By comparing the results in Tables \ref{E21}-\ref{E24} with the corresponding results in Tables \ref{E31}-\ref{E34}, we conclude that our schemes are robust with respect to the Poisson ratio $\nu$ and the hydraulic conductivity $K$.
\begin{table}[H]
\begin{center}
\caption{Errors and convergence rates given by Method 1 for Example 2 using $k=2$ and $l=1$ with $\nu = 0.49999$ and $K=10^{-6}$.}
\label{E31}
\centering
{\scriptsize
	\begin{tabular}{ccclclcl}
	\hline
	$h$  &  $\Delta t$  & \multicolumn{1}{l}{ $\bm{H}^1$ errors of $\bm{u}$} & \multicolumn{1}{l}{Orders} & \multicolumn{1}{l}{$L^2$ errors of $\xi$} & \multicolumn{1}{l}{Orders} & \multicolumn{1}{l}{$L^2$\& $H^1$ errors of $p$} & \multicolumn{1}{l}{Orders} \\ \hline
    1/2 & 1/4    & 1.712e+00 &  & 9.575e-02 &  & 2.071e-01 \& 9.962e-01 &  \\
    1/4 & 1/16   & 5.242e-01 & 1.71 & 7.503e-02 & 0.35 & 5.429e-02 \& 3.928e-01 & 1.93 \& 1.34 \\
    1/8 & 1/64   & 1.433e-01 & 1.87 & 1.186e-02 & 2.66 & 1.416e-02 \& 1.853e-01 & 1.94 \& 1.08 \\
    1/16 & 1/256 & 3.696e-02 & 1.95 & 2.478e-03 & 2.26 & 3.594e-03 \& 9.127e-02 & 1.98 \& 1.02 \\ \hline
	\end{tabular}
}
\end{center}
\end{table}
\begin{table}[H]
\begin{center}
\caption{Errors and convergence rates given by Method 1 for Example 2 using $k=3$ and $l=2$ with $\nu = 0.49999$ and $K=10^{-6}$.}
\label{E32}
\centering
{\scriptsize
	\begin{tabular}{ccclclcl}
	\hline
	$h$  &  $\Delta t$  & \multicolumn{1}{l}{ $\bm{H}^1$ errors of $\bm{u}$} & \multicolumn{1}{l}{Orders} & \multicolumn{1}{l}{$L^2$ errors of $\xi$} & \multicolumn{1}{l}{Orders} & \multicolumn{1}{l}{$L^2$\& $H^1$ errors of $p$} & \multicolumn{1}{l}{Orders} \\ \hline
    1/2 & 1/8     & 5.586e-01 &      & 1.466e-01 &      & 1.925e-02 \& 1.681e-01 &  \\
    1/4 & 1/64    & 8.014e-02 & 2.80 & 1.025e-02 & 3.84 & 2.906e-03 \& 5.269e-02 & 2.73 \& 1.67 \\
    1/8 & 1/512   & 9.895e-03 & 3.02 & 1.246e-03 & 3.04 & 3.759e-04 \& 1.351e-02 & 2.95 \& 1.96 \\
    1/16 & 1/4096 & 1.212e-03 & 3.03 & 1.416e-04 & 3.14 & 4.550e-05 \& 3.122e-03 & 3.05 \& 2.11 \\ \hline
	\end{tabular}
}
\end{center}
\end{table}
\begin{table}[H]
\begin{center}
\caption{Errors and convergence rates given by Method 2 for Example 2 using $k=2$ and $l=1$ with $\nu = 0.49999$ and $K=10^{-6}$.}
\label{E33}
\centering
{\scriptsize
	\begin{tabular}{ccclclcl}
	\hline
	$h$ & $\Delta t$                    & \multicolumn{1}{l}{ $\bm{H}^1$ errors of $\bm{u}$} & \multicolumn{1}{l}{Orders} & \multicolumn{1}{l}{$L^2$ errors of $\xi$} & \multicolumn{1}{l}{Orders} & \multicolumn{1}{l}{$L^2$\& $H^1$ errors of $p$} & \multicolumn{1}{l}{Orders} \\ \hline
    1/2 & 1/2   & 1.712e+00 &      & 9.575e-02 &      & 2.434e-01 \& 1.014e+00 &  \\
    1/4 & 1/4   & 5.242e-01 & 1.71 & 7.503e-02 & 0.35 & 6.470e-02 \& 4.003e-01 & 1.91 \& 1.34 \\
    1/8 & 1/8   & 1.433e-01 & 1.87 & 1.186e-02 & 2.66 & 1.680e-02 \& 1.866e-01 & 1.95 \& 1.10 \\
    1/16 & 1/16 & 3.696e-02 & 1.95 & 2.478e-03 & 2.26 & 4.253e-03 \& 9.144e-02 & 1.98 \& 1.03 \\
    \hline
	\end{tabular}
}
\end{center}
\end{table}
\begin{table}[H]
\begin{center}
\caption{Errors and convergence rates given by Method 2 for Example 2 using $k=3$ and $l=2$ with $\nu = 0.49999$ and $K=10^{-6}$.}
\label{E34}
\centering
{\scriptsize
	\begin{tabular}{ccclclcl}
	\hline
	$h$ & $\Delta t$                    & \multicolumn{1}{l}{ $\bm{H}^1$ errors of $\bm{u}$} & \multicolumn{1}{l}{Orders} & \multicolumn{1}{l}{$L^2$ errors of $\xi$} & \multicolumn{1}{l}{Orders} & \multicolumn{1}{l}{$L^2$\& $H^1$ errors of $p$} & \multicolumn{1}{l}{Orders} \\ \hline
    1/2 & 1/4    & 5.586e-01 &      & 1.466e-01 &      & 1.383e-02 \& 1.538e-01 &  \\
    1/4 & 1/16   & 8.014e-02 & 2.80 & 1.025e-02 & 3.84 & 2.314e-03 \& 5.242e-02 & 2.58 \& 1.55 \\
    1/8 & 1/64   & 9.895e-03 & 3.02 & 1.246e-03 & 3.04 & 2.675e-04 \& 1.351e-02 & 3.11 \& 1.96 \\
    1/16 & 1/256 & 1.212e-03 & 3.03 & 1.416e-04 & 3.14 & 2.800e-05 \& 3.123e-03 & 3.26 \& 2.11 \\
    \hline
	\end{tabular}
}
\end{center}
\end{table}

\section{Conclusions and outlook}
\label{conclusion}

In this paper, we present a priori estimates of the two fully discrete coupled schemes for the three-field formulation of Biot’s consolidation model. The theoretical results show that both schemes are unconditionally convergent with optimal error orders. We comment here that Method 2 achieves a second-order convergence in time without complicating the calculations. The detailed numerical experiments are carried out to verify the predictions of error estimates. In future work, we plan to develop some decoupled algorithms \cite{ju2020parameter, gu2022iterative} and the corresponding analysis based on the theory studied in this work.

\section{Acknowledgements}
The works of H. Gu and J. Li are partially supported by the NSF of China No. 11971221 and the Shenzhen Sci-Tech Fund No. RCJC20200714114556020, JCYJ20170818153840322 and JCYJ20190809150413261, and Guangdong Provincial Key Laboratory of Computational Science and Material Design No. 2019B030301001. The work of M. Cai is supported in part by NIH-BUILD grant through UL1GM118973, NIH-RCMI grant through U54MD013376, and the National Science Foundation awards (1700328, 1831950). The work of G. Ju is supported in part by the National Key R \& D Program of China (2017YFB1001604). 

%% The Appendices part is started with the command \appendix;
%% appendix sections are then done as normal sections
%% \appendix

%% \section{}
%% \label{}

%% If you have bibdatabase file and want bibtex to generate the
%% bibitems, please use
%%
%%  \bibliographystyle{elsarticle-num} 
%%  \bibliography{<your bibdatabase>}

%% else use the following coding to input the bibitems directly in the
%% TeX file.

\end{document}